\documentclass[12pt,reqno]{amsart}
\usepackage{amsmath,amssymb}
\usepackage[margin=1in]{geometry}
\usepackage{graphicx}
\usepackage[dvipsnames]{xcolor}
\usepackage[colorlinks=true,citecolor=blue]{hyperref}
\newcommand{\pa}{\partial}
\newcommand{\na}{\nabla}
\newcommand{\di}{\nabla\cdot}

\newcommand{\pnu}{\pa_\nu}

\newcommand{\intO}{\int_{\Omega}}

\newcommand{\dH}{d\mathcal{H}_x^{d-1}}

\renewcommand{\epsilon}{\varepsilon}

\newtheorem{theorem}{Theorem}[section]
\newtheorem{definition}{Definition}[section]
\newtheorem{lemma}{Lemma}[section]
\newtheorem{proposition}{Proposition}[section]

\newtheorem{remark}{Remark}[section]

\title[Chemotaxis-Navier-Stokes systems with Robin boundary conditions]{Global solutions for chemotaxis-Navier-Stokes system with Robin boundary conditions}

\author[M. Braukhoff]{Marcel Braukhoff}
\address{Institute for Analysis and Scientific Computing, Vienna University of Technology, Wiedner Hauptstrasse 8--10, 1040 Vienna, Austria}
\email{marcel.braukhoff@asc.tuwien.ac.at}
\author[B. Q. Tang]{Bao Quoc Tang}
\address{Institute of Mathematics and Scientific Computing, University of Graz, 
	Heinrichstrasse 36, 8010 Graz, Austria}
\email{quoc.tang@uni-graz.at}
\begin{document}
\maketitle
\begin{abstract}
	We consider a chemotaxis-Navier-Stokes system modelling cellular swimming in fluid drops where an exchange of oxygen between the drop and its environment is taken into account. This phenomenon results in an inhomogeneous Robin-type boundary condition. Moreover, the system is studied without the logistic growth of the bacteria population. We prove that in two dimensions, the system has a unique global classical solution, while the existence of a global weak solution is shown in three dimensions. In the latter case, we show that the energy is bounded uniformly in time. A key idea is to utilise a boundary energy to derive suitable {\it a priori} estimates. Moreover, we are able to remove the convexity assumption on the domain. 
\end{abstract}

\tableofcontents

\section{Introduction and Main results}
In recent years the analysis of pattern formation in biology  has become a thriving field, especially in the analysis of chemotaxis describing bacteria cells and their interaction with a chemical. In this paper, we study the following chemotaxis-Navier-Stokes system with signal consumption
\begin{equation}\label{C-NS}\left\{
\begin{aligned}
\partial_t n + u\cdot \na n - \Delta n &= \na\cdot (n\na c), &x\in\Omega, \; t>0,\\
\partial_t c+ u\cdot \na c - \Delta c &= -nc, &x\in\Omega, \; t>0,\\
\partial_t u - \mu \Delta u + \na\cdot(u\otimes u) &=\nabla P  -n\na \varphi,  &x\in\Omega, \; t>0,\\
\na\cdot u &= 0,  &x\in\Omega, \; t>0,
\end{aligned}\right.
\end{equation}
subject to boundary and initial data conditions
\begin{equation}\label{boundary}
\begin{aligned}
\na c \cdot \nu = \kappa(x) (\gamma(x) - c), \qquad \na n \cdot \nu = n\na c\cdot \nu, \qquad u = 0, \qquad x\in\Gamma, \; t>0\\
n(x,0) = n_0(x), \qquad c(x,0) = c_0(x), \qquad u(x,0) = u_0(x), \qquad x\in\Omega.
\end{aligned}
\end{equation}
Here $\Omega\subset \mathbb R^d$, $d=1,2,3$, is a bounded domain with smooth boundary $\Gamma:= \partial\Omega$, $\mu > 0$ is the viscosity, $\nu(x)$ is the unit outward normal vector at $x\in \Gamma$, and $\varphi$ is the gravitational potential. The explanation of the importance and the role of $\kappa ,\gamma:\Gamma\to \mathbb R_{\geq0}$ in the boundary condition \eqref{boundary} is explained in full detail in the next subsection.

\medskip
System \eqref{C-NS} (with slightly different boundary conditions) was introduced in \cite{tuval}  (see also, e.g., \cite{tao_bdoxygenconsumption}, or Sections 4.1 and 4.2 of the survey \cite{BBTW}). In this model, the population density of the bacteria is denoted by $n$, whereas $c$ stands for the chemical concentration. Assuming that the bacteria and the chemical are solved in an incompressible liquid like water, we use the Navier-Stokes equation for velocity $u$ to model its flow.  Due to the gravitational potential $\varphi$, the bacteria cells influence the liquid flow through their weight.    

\medskip
As an example, this model may be applied to describe the density $n$ of the species \textit{Bacillus subtilis} in a drop of water given by $\Omega$. Their otherwise random motion is known to be directed towards higher concentration $c$ of oxygen contraction, which they consume. 
In \cite{dombrowskietal,tuval}, one can experimentally observe that large coherent patterns emerge after some time, which became an interesting research topic in the mathematical community \cite{win_transAMS}. 
However, the rigorous results were devastating with respect to this matter: In order to facilitate the problem, usually the system was analyzed for homogeneous boundary conditions, i.e., $\kappa\equiv0$. On the one hand, it was shown that solutions subject to small initial data in a three dimensional domain combined with homogeneous Neumann boundary condition converged to the stationary, constant state $(\frac1{|\Omega|}\int_{\Omega}n_0,0,0)$. On the other hand, also every classical solution in two spacial dimensions converges to the same stationary state \cite{fan_zhao,jiang_wu_zheng,win_fluid_konvergenzresultat,zhang_li_decay}. Finally, the case was settled in  \cite{win_transAMS} that even ``eventual energy solution'' converge to the constant state.

\medskip
There are also different versions of the system of porous-medium type (see e.g. \cite{difrancesco_lorz_markowich}) or where the chemotaxis term is given in a more general form \cite{win_CalcVarPDE}. However, the long term behavior remains qualitatively the same - also without involving a fluid, see \cite{fan_jin,lswx,tao_win}.

\subsection{The boundary conditions}

All the previous mentioned articles have in common to use homogeneous boundary conditions. Nevertheless, in the experiments \cite{tuval}, the drop of water is  surrounded by air which leads inevitable to an oxygen exchange between the drop and the surroundings \cite{atkins}. Actually, already in the original paper \cite{tuval} introducing the model \eqref{C-NS}, the authors already use inhomogeneous Dirichlet conditions. 

\medskip
Therefore, let us have a closer look how to model the oxygen exchange and why this is crucial for the experiment. We follow the derivation of \cite{Bra17,BrLa19}. 
Assume that water is an ideal solvent for oxygen. The oxygen exchange at the boundary can be modeled using Raoul's law: On the one hand the amount of solving oxygen at $x\in\Gamma$ is proportional to the vapor pressure of the gaseous oxygen around $x$.  On the other hand, the outgoing rate of oxygen is proportional to  the concentration on the boundary, i.e., the rate of oxygen molecules leaving the drop at $x\in\Gamma$ is proportional to the number of molecules at $x$ (see \cite[Section 5.3, page 144]{atkins}). In order to have a closed system, we suppose that the oxygen vapor pressure is a given function. This is reasonable, because  the oxygen-diffusion coefficient  in air is three orders of magnitude larger than that in the	fluid \cite[page 2279]{tuval}. Moreover, the negligibility of the influence of the drop to the gaseous oxygen implies that the vapor pressure is constant in time. Adding both effects, we see that the oxygen-flux at the boundary is an affine  function of the concentration, which we write in the form
\begin{equation}\label{boundary2_0}
\begin{aligned}
\na c(x,t) \cdot \nu = \kappa(x) (\gamma(x) - c(x,t)), \qquad
x\in\Gamma, \; t>0
\end{aligned}
\end{equation}
for $\gamma,\kappa:\Gamma\to\mathbb{R}_{\geq0}$. This condition is also known as Henry's law in the context of sorption of chemicals to surfaces \cite{atkins}. Note that we do not want to assume that the drop is entirely surrounded by air, but also part of it can by connected to a solid exterior where there is no oxygen exchange. 
Therefore for on the solid--water interface we assume that $\kappa$ vanishes, which does not have to be the case on the water--air boundary.  For function $\gamma(x)$ (as in \cite{BrLa19}), one can interpret it as the maximal saturation of oxygen in the fluid. 
%
Note that for $\Omega$ being the ball and $\gamma$ and $\kappa$ being radially symmetric, one obtains Dirichlet boundary conditions (see e.g. \cite{tuval})
\[c(x)=\gamma(x)\qquad \text{for }x\in \partial \Omega\]
as a limit of \eqref{boundary2_0} for $\kappa\to\infty$, see \cite[Proposition 5.3]{BrLa19} for a proof of the stationary problem.

\medskip
Having changed the boundary condition for the oxygen concentration $c$, we need to adjust the boundary conditions for $n$ as well in order to preserve the mass of bacteria. For this we choose the no-flux conditions for $n$. In addition, we close the Navier-Stokes system with Dirichlet boundary conditions. Therefore the set of boundary conditions are given by
\begin{equation}\label{boundary2}
\begin{aligned}
\na c \cdot \nu = \kappa(x) (\gamma(x) - c), \qquad \na n \cdot \nu = n\na c\cdot \nu, \qquad u = 0, \qquad x\in\Gamma, \; t>0.
\end{aligned}
\end{equation}
In \cite{BrLa19}, the system \eqref{C-NS} combined with this boundary conditions is treated without the flow, i.e., $u=\nabla P = \nabla \varphi \equiv0$. Therein it is shown that if $\kappa\not\equiv0$ and $\gamma=const$ then \eqref{C-NS} and \eqref{boundary2} admit a unique stationary state for a given mass $\int_{\Omega} n dx$. Moreover $n$ and $c$ are positive but not constant. In the radial symmetric case, $n$ and $c$ are  even strictly convex. Up to the best of our knowledge, this is the only qualitative result for the system \eqref{C-NS} showing a non-trivial steady state.

\medskip
Let us mention related works on chemotaxis systems involving inhomogeneous boundary conditions. The articles \cite{chertock_etal_numeric,lee_kim_numerical_bioconvection,tuval} show numerically that models with inhomogenous boundary conditions match the experimental results. 
In \cite{Lorz}, a chemotaxis-fluid system with an inhomogeneous Dirichlet condition for $c$ on parts of the boundary is treated on a bounded two dimensional domain, and the local existence of weak solutions is shown therein. Recently, \cite{preprint_zhaoyin} imposes inhomogeneous Dirichlet condition on one side of the domain $\mathbb R^2\times(0,1)$. Under stronger technical assumptions on the consumption term,  \cite{preprint_zhaoyin} proves the existence and convergence of solutions for initial data being close to $(0,\gamma,0)$. Moreover, in spatial dimension one \cite{knosalla_global,knosalla_nadzieja_stationary} treat the related chemotaxis system 
\[
\begin{cases}
n_t=n_{xx} - (n E(c)_x)_x,\\
c_t=c_{xx} - n E(c)
\end{cases}
\]
for either a inhomogeneous Dirichlet or Neumann conditions. Here, $E$ satisfies $E(c)\to 0$ for $c\to 0$ and $c\to \infty$. The existence of global, bounded solutions is proved in \cite{knosalla_global}, whereas \cite{knosalla_nadzieja_stationary} proves the existence and uniqueness of the stationary state.

\subsection{Global existence vs.\ logistical source}
The first analytical results for system \eqref{C-NS}--\eqref{boundary} with logistic growth of the density, i.e. the equation for $n$ is replaced by 
\begin{equation}\label{logistic}
\partial_t n + u\cdot \na n - \Delta n = \na\cdot(n\na c) + n(1-n)
\end{equation}
were delivered in a paper of the first author \cite{Bra17}, in which the global existence of classical and weak solutions was shown in two and three dimensions, respectively.
\medskip

The global existence of solutions to \eqref{C-NS} with homogeneous boundary conditions crucially depends on the energy functional
\begin{equation}\label{normal-energy}
	S(t) = \int_{\Omega}n(t)\log n(t)dx + a\int_{\Omega}\left|\na\sqrt{c(t)} \right|^2dx + b\int_{\Omega}|u(t)|^2dx
\end{equation}
which is decreasing for suitable  constants $a, b>0$, see e.g. \cite{Win12,Win16}. This gives the necessary a-priori estimates to start the bootstrapping, which eventually leads to global (strong, weak) solutions. In the case of Robin-type boundary conditions (for the oxygen $c$), this strategy is not directly applicable since the functional $S(t)$ fails to decrease in time because of the boundary terms in the estimate. This problem was solved in \cite{Bra17}, firstly by transforming \eqref{C-NS} into homogeneous Neumann boundary conditions, and secondly, to cope with the extra terms coming from the transformation, by introducing the logistical growth term as in \eqref{logistic}. The logistic term gives a bound in $L^2(0,T;L^2(\Omega))$ for free just by integrating \eqref{logistic} on $\Omega\times (0,T)$. This estimate can then be used in an essential way in a bootstrap argument to get global solutions.


\medskip
The logistic nonlinearity acts as a damping term 
and therefore it usually helps in the analysis of chemotaxis systems. For instance, under homogeneous boundary conditions, system \eqref{C-NS} with a logistic growth is very well studied in \cite{lankeit_m3as} in which global weak solutions were shown to be smooth after some positive time. Moreover, convergence of solutions to the steady state $(1,0,0)$ was also proved. Similar results were obtained in \cite{lankeit_wang} for the case without fluids and in \cite{win_nutrienttaxis}  in the case of food-supported proliferation. A recent study \cite{Miz19} demonstrates well the effect of logistic growth (together with nonlinear diffusion) to the well-posedness of chemotaxis systems. We however remark that a logistic growth term might lead to interesting new effects in chemotaxis \cite{lankeit_thresholds,hillenpainter_spatiotemporalchaos,win_transient}. For example, one can easily see that the mass of the bacteria is no longer conserved if a logistic source term is added to the first equation. 

\medskip
The global well-posedness of the chemotaxis-Navier-Stokes system without logistic growth \eqref{C-NS} together with the inhomogeneous boundary conditions \eqref{boundary} is therefore a challenging problem, and it is the main aim of the present paper.
\subsection{Key ideas}
As mentioned in the previous subsection, due to inhomogeneous boundary conditions \eqref{boundary}, the usual energy \eqref{normal-energy} is not decreasing in time along a trajectory of \eqref{C-NS}. Moreover, the lack of the logistic growth also seems to break the strategy of transforming \eqref{C-NS}--\eqref{boundary} into a system with homogeneous boundary conditions. Our key idea to deal with this issue is first to introduce {\it a boundary energy} of the form
\begin{equation*}
	S^{\text{boundary}}(t):= \int_{\Gamma} \kappa(x)\left[\gamma(x)\log\frac{\gamma(x)}{c(x,t)} - \gamma(x) + c(x,t)\right]\dH,
\end{equation*}
and then to look at the evolution of the {\it total energy}
\begin{equation*}
\mathcal F(t) = S(t) + S^{\text{boundary}}(t)
\end{equation*}
with $a = 2$ and $b = K$, for a sufficiently large constant $K$. We will show that this total energy satisfies
\begin{equation}\label{en1}
	\frac{d}{dt}\mathcal F(t) \leq p\mathcal F(t) + q
\end{equation}
for some constant $p, q>0$, which consequently leads to a set of a-priori estimates. These estimates are enough in two dimension to start a bootstrap argument to obtain global classical solutions, while they ensure an approximating procedure in three dimensions to get global weak solutions. 

\medskip
As one can see from \eqref{en1} that though the solution is global, the energy might grow exponentially. To show that the total energy $\mathcal{F}(t)$ is in fact bounded uniformly in time, we introduce yet another energy functional
\begin{equation*}
	S^{\text{add}}(t):= \int_{\Omega}\left[c(x,t)\log \frac{c(x,t)}{\widehat{\gamma}(x)} - c(x,t) + \widehat{\gamma}(x)\right]dx
\end{equation*}
where $\widehat{\gamma}$ is a smooth extension of $\gamma$ to $\overline{\Omega}$. Now by considering $\mathcal{F}^{\text{new}}(t) = \mathcal{F}(t) + LS^{\text{add}}(t)$ for some suitable constant $L>0$, we obtain
\begin{equation}\label{en2}
	\frac{d}{dt}\mathcal{F}^{\text{new}}(t) + \lambda \mathcal{F}^{\text{new}}(t) \leq C
\end{equation} 
for some $\lambda, C>0$. This inequality gives the uniform-in-time bound for $\mathcal{F}^{\text{add}}$ and eventually the desired bound for the total energy $\mathcal{F}$.

\medskip
We also would like to emphasize that we do not assume the domain $\Omega$ to be convex. The convexity of $\Omega$ was very useful in the literature when dealing with the analysis of \eqref{C-NS}, see e.g. \cite{Win12,Win16}. Though it is natural to assume that a fluid drop has a convex shape, there exist situations when it is not the case, for instance, when the drop is in contact with an uneven surface. In \cite{lankeit_m3as,MS14}, the authors were also able to remove this technical condition on the convexity of $\Omega$ by using the boundedness of the domain curvature (see \cite[Lemma 4.2]{MS14}). Our main idea is to go one step further and use the full power of the dissipation terms arising from the diffusion of the oxygen (see the proof of Lemma \ref{lem:energy2}).

\subsection{Main Results}
We begin with definitions of classical and weak solutions.
\begin{definition}[Classical solutions]
	A quadruplet $(n,c,u,P)$ is called a classical solution to \eqref{C-NS}--\eqref{boundary} on $(0,T)$ if
	\begin{align*}
	n,c&\in C^{2+2\delta,1+\delta}\left(\overline\Omega\times(0,T)\right)\cap C^{0}\left(\overline\Omega\times[0,T)\right),
	\\
	u&\in C^{2+2\delta,1+\delta}\left(\overline\Omega\times(0,T)\right)\cap C^{0}\left(\overline\Omega\times[0,T)\right),
	\\
	P&\in C^{1+\delta,\delta}\left(\Omega\times(0,T)\right),
	\end{align*}
	for some $\delta>0$, and the equations in \eqref{C-NS}--\eqref{boundary} are satisfied pointwise.	
\end{definition}

\begin{definition}[Weak solutions]\label{weak_sol}
	A triple $(n,c,u)$ is called a global weak solution of \eqref{C-NS}--\eqref{boundary} if
	\begin{equation*}
		n\in L^1_{loc}([0,\infty);W^{1,1}(\Omega)), \quad c\in L^1_{loc}([0,\infty); W^{1,1}(\Omega)), \quad u \in L^1_{loc}([0,\infty); W^{1,1}_0(\Omega;\mathbb R^3))
	\end{equation*}
	such that $n\geq 0$ and $c\geq 0$ a.e. in $\Omega\times (0,\infty)$,
	\begin{equation*}
		nc \in L^1_{loc}(\Omega\times [0,\infty)), \quad u\otimes u \in L^1_{loc}(\Omega\times[0,\infty); \mathbb R^{3\times 3}), \qquad \text{ and }
	\end{equation*}
	\begin{equation*}
		n\na c, \quad nu, \quad cu \quad \text{ belong to } \quad L^1_{loc}(\Omega\times [0,\infty); \mathbb R^3),
	\end{equation*}
	that $\na\cdot u = 0$ a.e. in $\Omega\times (0,\infty)$, and that
	\begin{align*}
		-\int_0^\infty\int_{\Omega} n\partial_t \psi dxdt = \int_{\Omega}n_0\psi(\cdot,0)dx\\ -\int_0^\infty\int_{\Omega}\na n\cdot \na \psi dxdt + \int_0^\infty\int_{\Omega}n\na c \cdot \na \psi dxdt + \int_0^\infty\int_{\Omega}nu\cdot \na\psi dxdt,
	\end{align*}
	\begin{align*}
		-\int_0^\infty\int_{\Omega}c\partial_t\psi dxdt = \int_{\Omega}c_0\psi(\cdot,0)dx
		-\int_0^\infty\int_{\Omega}\na c\cdot \na\psi dxdt\\ + \int_0^\infty\int_{\Gamma}\kappa (\gamma - c)\psi \dH dt - \int_0^\infty\int_{\Omega}nc\psi dxdt + \int_0^\infty\int_{\Omega} cu\cdot \na\psi dxdt
	\end{align*}
	for all $\psi \in C_0^\infty(\Omega\times [0,\infty))$, and
	\begin{align*}
		-\int_0^\infty\int_{\Omega} u\cdot \xi_t dxdt = \int_{\Omega}u_0\cdot \xi(\cdot,0)dx\\
		-\int_0^\infty\int_{\Omega} \na u\cdot \na \xi dxdt + \int_0^\infty\int_{\Omega} u\otimes u\cdot \na \xi dxdt - \int_0^\infty\int_{\Omega}n\na \varphi \cdot \xi dxdt
	\end{align*}
	for all $\xi \in C_0^\infty(\Omega\times[0,\infty); \mathbb R^3)$ satisfying $\na\cdot \xi \equiv 0$.
\end{definition}

As usual, we denote by
\begin{equation*}
	L^2_\sigma(\Omega):=\overline{D_\sigma(\Omega)}^{\|\cdot\|_{L^2(\Omega)}}, \quad \text{ where } \quad D_\sigma(\Omega):= \{u\in C_0^\infty(\Omega)^d:\nabla\cdot u =0\}, 
\end{equation*}

and let $\mathcal P^\infty$ be the Helmholz projection $L^2(\Omega)^d \to L^2_\sigma(\Omega)^d$. We denote by
\begin{equation*}
	A: D(A) \subset L_\sigma^2(\Omega) \to L_\sigma^2(\Omega), \quad Au:= -\mathcal P^\infty \Delta u
\end{equation*}
the Stokes operator with Dirichlet boundary conditions, where the domain of $A$ is given by
\begin{equation*}
D(A) = L^2_\sigma(\Omega)^d\cap H_0^1(\Omega)^d \cap H^2(\Omega)^d.
\end{equation*} The main results of this paper are the following two theorems.
\begin{theorem}[Global classical solutions in dimension two]\label{thm:main2D}
	Let $d \leq 2$ and assume that 
	the data satisfies
	\begin{equation*}
		0< \kappa, \gamma \in C^1(\Gamma), \quad \varphi \in C^1(\overline{\Omega}).
	\end{equation*}	
	Then for any initial data $(n_0, c_0, u_0)$ satisfying
	\begin{equation*}
	\left\{\begin{aligned}
	0< n_0&\in C^0(\overline{\Omega})\cap H^1(\Omega),\\
	0< c_0&\in W^{1,10}(\Omega),\\
	u_0&\in D(A^{\alpha}) \qquad \text{for some } \frac d4<\alpha<1.\end{aligned}\right\},
	\end{equation*}
	there exists a unique global classical solution to \eqref{C-NS}--\eqref{boundary}.
\end{theorem}

\begin{theorem}[Global weak solutions in three dimensions]\label{thm:main3D}
	Let $d=3$, and assume that 
	\begin{equation}\label{varphi}
		\varphi \in W^{1,\rho}(\Omega), \quad \text{ for some } \quad \rho > 6, 
	\end{equation}
	and
	\begin{equation}\label{cond_data}
		\sqrt{\kappa}\in H^1(\Gamma)\cap L^\infty(\Gamma), \quad 0 < \underline{\gamma} \leq \gamma, \quad \text{ and } \quad \sqrt{\gamma} \in H^1(\Omega)\cap L^\infty(\Gamma).
	\end{equation}
	Then for any initial data $(n_0, c_0, u_0)$ satisfying
	\begin{equation*}
	\begin{aligned}
		& n_0 >0 \quad \text{ and } \quad \int_{\Omega}n_0\log n_0 dx < +\infty,\\
		& 0 < c_0 \in L^\infty(\Omega) \quad \text{ and } \quad \sqrt{c_0} \in H^1(\Omega),\\
		& u_0 \in L^2_{\sigma}(\Omega),
	\end{aligned}
	\end{equation*}
	the system \eqref{C-NS}--\eqref{boundary} has a global weak solution. Moreover, the global energy is bounded uniformly in time, i.e.
	\begin{equation*}
			\sup_{t\in [0,\infty)}\left(\int_{\Omega}n(t)\log n(t)dx + \| \na\sqrt c(t)\|_{L^2(\Omega)}^2 + \|u(t)\|_{L^2(\Omega)}^2\right) \leq C
	\end{equation*}
	where $C$ depends only on initial energy, on the data $\mu, \kappa, \gamma,  \varphi$, and on the domain $\Omega$.
\end{theorem}

\begin{remark}[Extensions]
	We believe that our approach is extendable to a more general system than \eqref{C-NS}--\eqref{boundary}, for instance
	\begin{equation}\label{C-NS-extension}
	\left\{
	\begin{aligned}
	\partial_t n + u\cdot \na n - \Delta n &= \na\cdot (n\chi(c)\na c), &x\in\Omega, \; t>0,\\
	\partial_t c+ u\cdot \na c - \Delta c &= -nf(c), &x\in\Omega, \; t>0,\\
	\partial_t u - \mu \Delta u + \na\cdot(u\otimes u) &=\nabla P  -n\na \varphi,  &x\in\Omega, \; t>0,\\
	\na\cdot u &= 0, &x\in\Omega, \; t>0,
	\end{aligned}
	\right.
	\end{equation}
	for some functions $\chi$ and $f$ satisfying suitable conditions (see e.g. \cite{Win16} for the case with homogeneous boundary conditions), though non-trivial modifications need to be carried out. We leave this interesting open issue for the interested reader.
\end{remark}

\medskip
{\bf The rest of this paper is organized as follows:} In the next section, we consider approximate systems of \eqref{C-NS} and derive necessary \textit{a priori} estimates. Using these estimates, we prove the main theorems in Section \ref{proofs}.

\medskip

{\bf Notation:} In this paper, we will use the following notation:
\begin{itemize}
	\item We will denote by $C$ a generic constant {\it independent of time}, which can be different from line to line, or even in the same line. When a constant depends on the time horizon $T>0$, we will write $C_T$ instead.
	\item For any $T>0$ and $1\leq p\leq \infty$, we denote by $Q_T:= \Omega\times (0,T)$ and
	\begin{equation*}
		L^p(Q_T):= L^p(0,T;L^p(\Omega))
	\end{equation*}

	with the usual norm
	\begin{equation*}
		\|f\|_{L^p(Q_T)}:= \left(\int_0^T\int_{\Omega}|f|^pdxdt \right)^{\frac 1p}
	\end{equation*}
	when $p<\infty$ and 
	\begin{equation*}
		\|f\|_{L^\infty(Q_T)}:= \text{ess sup}_{t\in (0,T)}\|f(t)\|_{L^\infty(\Omega)}.
	\end{equation*}
\end{itemize}

\section{Approximate systems and a-priori estimates}
If the evolution equation for the density $n$ in \eqref{C-NS}--\eqref{boundary} is replaced by
\[\partial_t n + u\cdot \nabla n - \Delta n = \nabla\cdot(n\nabla c) + n(1-n), \]
then the local existence of a classical solution was done in \cite[Proposition 2.6]{Bra17} by a standard fixed point argument. It is remarked that the proof of this result does not use any structural of the logistic growth $n(1-n)$, and it is therefore also applicable to \eqref{C-NS}--\eqref{boundary}. 
For the reader's convenience we recall Proposition 2.6 from \cite{Bra17} (without the logistic term). 
\begin{proposition}\label{a-prop-local-solution-hom}
	Let $d\in\{1,2,3\}$ and $\Omega\subset\mathbb{R}^d$ be a bounded domain with smooth boundary. 
	
	Then there exists a maximal $T_{\max} \in(0,\infty]$ such that \eqref{C-NS}-\eqref{boundary} possesses a classical solution on $(0,T)$ for every $0<T<T_{\max}$ with $n \geq 0$ and $c\geq 0$.
	Furthermore, if 
	\begin{equation}\label{a-formula-explosion-condition-transformed}
	\limsup_{t\uparrow T_{\max}}\left( \|n(t)\|_{L^\infty(\Omega)}+\|\na n (t)\|_{L^2(\Omega)}+\|c(t)\|_{W^{1,4}(\Omega)}+\|A^\alpha u(t)\|_{L^2(\Omega)}\right) < +\infty
	\end{equation}
	then $T_{\max} = \infty$. The solution $(n,c,u,P)$ is unique up to a constant for $P$.
\end{proposition}
In case $d=3$, as we do not expect to prove the existence of a global classical to the Navier-Stokes equation, we aim for weak solutions. Therefore, we consider in this case the following approximating sequence for $\epsilon\geq0$ and $m\in \mathbb N\cup\{\infty\}$.
\begin{equation}\label{Smod}\left\{
\begin{aligned}
&\pa_t n^{\epsilon,m} +u^{\epsilon,m}\cdot\nabla n^{\epsilon,m}- \Delta n^{\epsilon,m} = \di(n^{\epsilon,m}\na c^{\epsilon,m})+\epsilon n^{\epsilon,m}(1-(n^{\epsilon,m})^2), &x\in\Omega,\\
&\pa_t c^{\epsilon,m} +u^{\epsilon,m}\cdot\nabla c^{\epsilon,m}- \Delta c^{\epsilon,m} = -n^{\epsilon,m}c^{\epsilon,m}, &x\in\Omega,\\
&\pa_tu^{\epsilon,m} =-Au^{\epsilon,m}-\mathcal P^m[\nabla (u^{\epsilon,m}\otimes u^{\epsilon,m})+n^{\epsilon,m}\nabla\varphi], &x\in \Omega,\\
&\pnu c^{\epsilon,m} = \kappa(x) (\gamma(x)-c^{\epsilon,m}), &x\in\Gamma,\\
&\pnu n^{\epsilon,m} = n^{\epsilon,m}\pnu c^{\epsilon,m}, &x\in\Gamma,\\
&c^{\epsilon,m}(x,0) = c_0^{\varepsilon,m}(x), \; n^{\epsilon,m}(x,0) = n_0^{\varepsilon,m}(x), \; u^{\epsilon,m}(0)=u_0^{\varepsilon,m} &x\in\Omega,
\end{aligned}\right.
\end{equation} 
where 
\begin{equation}\label{a-initial-functions}
\left\{\begin{aligned}
0< n_0^{\varepsilon,m}&\in C^0(\overline{\Omega})\cap H^1(\Omega),\\
0< c_0^{\varepsilon,m}&\in W^{1,10}(\Omega),\\
u_0^{\varepsilon,m}&\in D(A^{\alpha}) \qquad \text{for some } \frac d4<\alpha<1,\end{aligned}\right. 
\end{equation}
and
\begin{equation*}
\lim_{\varepsilon\to 0}\sup_{m\in \mathbb N\cup \{\infty\}}\left(\|n_0^{\epsilon,m} - n_0\|_{L^1(\Omega)} + \|c_0^{\epsilon,m} - c_0\|_{L^\infty(\Omega)} + \|u_0^{\epsilon,m} - u_0\|_{L^2(\Omega)}\right) = 0,
\end{equation*}
\begin{equation*}
\sup_{\varepsilon>0, m\in \mathbb N\cup \{\infty\}}\left[\int_{\Omega}n_0^{\varepsilon,m}\log n_0^{\varepsilon,m}dx + \|\sqrt{c_0^{\varepsilon,m}}\|_{H^1(\Omega)}\right] <+\infty.
\end{equation*}

Here, $\mathcal P^m$ denotes the Leray projection onto the space of the first $m$ eigenvectors of $A$. 
For any fixed $\varepsilon>0$ and $\mathbb N\ni m < \infty$, there exists a global classical solution $(n^{\varepsilon,m}, c^{\varepsilon,m}, u^{\varepsilon,m})$ to \eqref{Smod} with $n^{\varepsilon,m}\geq 0$ and $c^{\varepsilon,m} \geq 0$ (see \cite[Proposition 4.6]{Bra17}).

\begin{remark}\label{rem.equivalent}
	The systems \eqref{C-NS}--\eqref{boundary} and  \eqref{Smod} are equivalent for $\epsilon=0$ and $m=\infty$ (see \cite[Theorems 1.7, 7.5 and 7.6]{kato}).
\end{remark}

The general strategy to study global existence of solutions for \eqref{C-NS} is the following: when $d\in \{1,2\}$, we show that the local solution obtained in Proposition \ref{a-prop-local-solution-hom} satisfies the criterion \eqref{a-formula-explosion-condition-transformed}, whence its global existence; while in the case $d = 3$, we prove that as $\varepsilon\to 0$ and $m \to \infty$, the global classical solution to the approximate system \eqref{Smod} converges to a global weak solution of \eqref{C-NS}. In both cases, we will use the same {\it a priori} estimates for either the local solution in Proposition \ref{a-prop-local-solution-hom} or the global solution of the approximate system \eqref{Smod}. Therefore, for the rest of this paper, we use a fixed (but arbitrary) time horizon $T$ with $0<T<T_{\max}$ when dealing with the former solution, while $0<T<\infty$ when dealing with the latter. To avoid complicated notation we will, in this section, suppress the superscript $\epsilon$ and $m$ in the solution of \eqref{Smod}, and write it simply $(n,c,u)$. Moreover, the generic constants $C>0$, which we use frequently, do not depend on $\epsilon$ nor on $m$.


\medskip
We start with the following immediate estimates, which will be useful in the sequel analysis.
\begin{lemma}\label{L1Linf}
	For all $t\in (0,T)$,
	\begin{equation*}
	\|n(t)\|_{L^1(\Omega)} \leq C \quad \text{ and } \quad \|c(t)\|_{L^\infty(\Omega)} \leq \max\left\{\|\gamma\|_{L^\infty(\Gamma)};  \|c_0\|_{L^\infty(\Omega)}\right\}.
	\end{equation*}
\end{lemma}
\begin{proof}
	The $L^\infty$-estimate of $c$ follows from the maximum principle. For the estimate of $n$ we integrate the equation of $n$ in \eqref{Smod} and use the incompressibility $\na\cdot u = 0$ as well as the boundary condition $\partial_\nu n = n\partial_\nu c$ to get
	\begin{equation}\label{nl}
	\partial_t \int_{\Omega} n dx + \epsilon\int_{\Omega}n^3dx = \epsilon\int_{\Omega}ndx \leq \epsilon\int_{\Omega}\left(n^3 - n + \frac{4\sqrt{6}}{9}\right)dx
	\end{equation}
	thanks to the non-negativity of $n$. Thus
	\begin{equation*}
	\partial_t \int_{\Omega} n dx + \epsilon \int_{\Omega} n dx \leq \frac{4\sqrt{6}}{9}|\Omega|\epsilon.
	\end{equation*}
	Hence
	\begin{equation*}
	\int_{\Omega}n(x,t)dx \leq e^{-\epsilon t}\int_{\Omega}n_0(x)dx + \frac{4\sqrt{6}}{9}|\Omega|(1 - e^{-\epsilon t}) \leq \|n_0\|_{L^1(\Omega)} + \frac{4\sqrt{6}}{9}|\Omega|.
	\end{equation*}
	which gives the desired estimate for $n$ since $n$ is non-negative.
\end{proof}

\medskip
We are going to use the following functions
\begin{equation}\label{entropy}
s(y):=y\log y -y+1 \qquad \text{ and } \qquad s^\infty(y|z):=y\log \frac{y}{z} -y+z.
\end{equation}
\begin{lemma}\label{lem.ee.n}
	We have the following identity for all $t\in (0,T)$
	\begin{equation*}
	\frac d{dt}\int_{\Omega}s(n)dx+4\int_{\Omega}\left|\nabla\sqrt{n}\right|^2 dx+\epsilon\int_{\Omega}n(1-n^2)\log ndx
	=\int_{\Omega}\nabla c\cdot \nabla  ndx
	\end{equation*}
\end{lemma}
\begin{proof}
	
Using the equation for $n$ and integration py parts, we directly see that
	\begin{align*}
	\frac d{dt}\int_{\Omega}s(n)dx	&=
	\int_{\Omega}\partial_tn\log{n}dx
	\\&= \int_{\Omega}(\Delta n-\nabla\cdot (n\nabla c))\log ndx-\int_{\Omega}u\cdot\underbrace{\nabla n\log{n}}_{\nabla s(n)}dx-\epsilon\int_{\Omega}n(1-n^2)\log ndx
	\\&= -\int_{\Omega}(\nabla n-n\nabla c)\cdot \nabla\log{n}dx+\int_{\Omega}\underbrace{\nabla\cdot u}_{=0} s(n)dx-\epsilon\int_{\Omega}n(1-n^2)\log ndx
	\\&=-	\int_{\Omega}\frac{|\nabla{n}|^2}n dx-\epsilon\int_{\Omega}n\log n(1-n^2)dx
	+\int_{\Omega}\nabla c\cdot \nabla ndx,
	\end{align*}
	where we have used $\nabla n\cdot \nu = n\nabla c \cdot \nu$ and $u = 0$ on $\Gamma$.
\end{proof}
\begin{lemma}\label{lem.entropy.equality.part2}
	The following identity holds 
	\begin{multline}\label{p0}
	\partial_t\int_{\Omega}|\nabla\sqrt{c}|^2dx 
+\frac12 \int_{\Omega}|\nabla^2\log c|^2cdx
+\int_\Omega|\nabla \sqrt{c}|^2ndx
\\=\int_{\Omega}\Delta(|\nabla\sqrt{c}|^2)dx 
- \frac12\int_{\Omega}\nabla n\cdot\nabla cdx-2\int_\Omega \nabla\sqrt c\cdot\nabla ^Tu\nabla \sqrt{c}dx.
\end{multline}
\end{lemma}
\begin{proof}
	To prove Lemma \ref{lem.entropy.equality.part2}, we first derive the equation of $\sqrt{c}$. In $\Omega$ we have
	\begin{align*}
	\partial_t \sqrt c+u\cdot\nabla \sqrt{c}-\Delta \sqrt c &=
	\frac{1}{2\sqrt c}(\partial_t  c+u\cdot \nabla c-\Delta c )+\frac{|\nabla c|^2}{4\sqrt{c}^3}
	\\&=-\frac{1}{2\sqrt{c}}nc+\frac{|\nabla\sqrt c|^2}{\sqrt c},
	\end{align*}
and on $\Gamma$ 
\begin{align}\label{boundary_sqrt_c}
\partial_\nu\sqrt c=\frac{1}{2\sqrt c}\partial_\nu c=\frac{\kappa }{2}\left(\frac{\gamma}{\sqrt{c}}-\sqrt{c}\right)
\end{align}
	From that, we can calculate
	\begin{equation}\label{p1}
	\begin{aligned}
	\partial_t|\nabla\sqrt{c}|^2
	&=
	2\nabla \partial_t \sqrt{c}\cdot \nabla \sqrt{c} 
	\\&=
	2\nabla \Delta \sqrt{c}\cdot \nabla \sqrt{c}
	+2\nabla \frac{|\nabla\sqrt{c}|^2}{\sqrt{c}}\cdot \nabla \sqrt{c}
	-\nabla(\sqrt{c} n )\cdot \nabla \sqrt{c} -2\nabla (u\cdot \nabla\sqrt c)\cdot\nabla \sqrt{c}
	\\&=
	\Delta |\nabla\sqrt{c}|^2
	-2 |\nabla^2\sqrt{c}|^2
	+2\nabla \frac{|\nabla\sqrt{c}|^2}{\sqrt{c}}\cdot \nabla \sqrt{c}
		\\&
		\qquad 
	-n |\nabla \sqrt{c}|^2-\nabla n\cdot\sqrt{c}\nabla \sqrt{c}-2\nabla\sqrt c\cdot\nabla ^Tu\nabla \sqrt{c}-2u\cdot \nabla^2\sqrt c\nabla \sqrt{c}
	\end{aligned}
\end{equation}
	using
	\begin{align*}
	2\nabla \Delta \sqrt{c}\cdot \nabla \sqrt{c}
	&=\Delta |\nabla\sqrt{c}|^2-2|\nabla^2\sqrt{c}|^2.
	\end{align*}
	For the third term on the right hand side of \eqref{p1}, we compute
	\begin{align*}
	&2\nabla \frac{|\nabla\sqrt{c}|^2}{\sqrt{c}}\cdot \nabla \sqrt{c}
	\\&=
	2 (\nabla(\sqrt{c})^{-1}|\nabla\sqrt{c}|^2
	)\cdot \nabla \sqrt{c} 
	+2 (\frac{1}{\sqrt{c}}\nabla|\nabla\sqrt{c}|^2
	)\cdot \nabla \sqrt{c}
	\\&=-\frac{2}{c}|\nabla\sqrt{c}|^4
	+\frac{4}{\sqrt{c}}\nabla\sqrt{c}\cdot \nabla^2\sqrt{c}
	\nabla \sqrt{c} .
	\end{align*}
	Moreover, $\nabla\cdot u=0$ implies
	\begin{align*}
	-2u\cdot \nabla^2\sqrt c\nabla \sqrt{c} &= -u\cdot \nabla|\nabla \sqrt{c}|^2=-\nabla\cdot (u|\nabla\sqrt c|^2)
	\end{align*}
	Inserting these computations into \eqref{p1} leads to
	\begin{align*}
	\partial_t&|\nabla\sqrt{c}|^2 
	+2 |\nabla^2\sqrt{c}|^2-\frac{4}{\sqrt{c}}\nabla\sqrt{c}\cdot \nabla^2\sqrt{c}
	\nabla \sqrt{c} +\frac{2}{c}|\nabla\sqrt{c}|^4
	\\&=
	\Delta |\nabla\sqrt{c}|^2-\nabla\cdot (u|\nabla\sqrt c|^2)
	-|\nabla\sqrt{c}|^2 n-\frac12 \nabla n\cdot\nabla c-2\nabla\sqrt c\cdot\nabla ^Tu\nabla \sqrt{c}.
	\end{align*}
	From the binomial formula for matrices, it follows that
	\begin{multline*}
	|\nabla^2\sqrt{c}|^2-2\nabla\sqrt{c}\cdot\nabla^2\sqrt{c}\nabla\sqrt{c} +\frac{1}{c}|\nabla\sqrt{c}|^4
	\\=
	\left|\nabla^2\sqrt{c}-\frac{1}{\sqrt{c}}\nabla\sqrt{c}\otimes\nabla\sqrt{c}\right|^2
	=\left|\sqrt{c}\nabla\left(\frac{\nabla \sqrt{c}}{\sqrt{c}}\right)\right|^2=c|\nabla^2\log\sqrt{c}|^2 .
	\end{multline*}
	Therefore, after an integration over $\Omega$, we have
	\begin{align*}
	&	\partial_t\int_{\Omega}|\nabla\sqrt{c}|^2dx 
	+\frac12 \int_{\Omega}c|\nabla^2\log c|^2dx
	+\int_\Omega|\nabla \sqrt{c}|^2ndx
	\\&\qquad=\int_{\Omega}\Delta(|\nabla\sqrt{c}|^2)dx -\int_{\Omega}\nabla\cdot (u|\nabla\sqrt c|^2)dx
	- \frac12\int_{\Omega}\nabla n\cdot\nabla cdx-2\int_\Omega \nabla\sqrt c\cdot\nabla ^Tu\nabla \sqrt{c}dx.
	\end{align*}
	We observe that the second term on the r.h.s.\ vanishes, because of the Gau\ss{} formula and the fact that $u=0$ on $\Gamma$.
\end{proof}
To estimate the last term on the right-hand side of \eqref{p0}, we need the following lemma.
\begin{lemma}\label{lem.bernstein}
	The inequality 		
	\begin{align*}
	\frac14\intO\frac{|\na  c|^4}{ c^3}\leq \int_{\pa\Omega}|\na \log c| ^2\kappa \left(\gamma- c\right)d\mathcal{H}^{d-1}_x+(2+d)\intO  c|\nabla^2\log c|^2,
	\end{align*}
	holds for any smooth function $c$ satisfying $\partial_\nu c = \kappa (\gamma - c)$ on $\Gamma$.
\end{lemma}
\begin{proof}
	 We follow the ideas from \cite[Lemma 3.3]{Win12}. We compute
	\begin{align*}
	\intO \frac{|\na c|^4}{c^3}dx &=\intO |\na \log c| ^2\na \log c\cdot \na c\, dx.
	\end{align*}
	Using integration by parts and the boundary conditions for $c$, we have
	\begin{align*}
	\intO |\na \log c| ^2\na \log c\cdot \na c\,dx
	&=
	\int_{\pa\Omega}|\na \log c| ^2\pa_\nu c\, d\mathcal{H}^{d-1}_x
	\\&\quad	-
	\intO \na|\na \log c| ^2\cdot(\na \log c)  c\,dx
	-
	\intO |\na \log c| ^2(\Delta \log c) c\,dx
	\\&= 
	\int_{\pa\Omega}|\na \log c| ^2 \kappa \left(\gamma- c\right)\mathcal{H}^{d-1}_x
	\\&\qquad-2\intO \frac{1}{ c}(\na^2\log c\na c)\cdot\na c
	dx
	-\intO\frac{|\na c|^2}{ c}\Delta \log c\,dx
	\end{align*}
	using $\nabla |\nabla \log c|^2=2\nabla^2 \log c\nabla \log c$. By Young's inequality, we see that
	\begin{align*}
	-2\intO \frac{1}{ c}(\na^2\log c\na c)\cdot\na c
	dx
	&\leq \frac12\intO\frac{|\na  c|^4}{ c^3}+2\intO  c|\nabla^2\log c|^2
	\end{align*}
	and 
	
	\begin{align*}
	-\intO\frac{|\na c|^2}{ c}\Delta \log c\,dx
	&\leq \frac14\intO\frac{|\na  c|^4}{ c^3}+\intO  c|\Delta\log c|^2.
	\end{align*}
	Note that we have the fundamental estimate $|\Delta \log c|^2=|\mathrm{trace}(\na^2\log c)|^2\leq d|\na^2\log c|^2$. Collecting all the previous calculations and estimates yields
	\begin{align*}
	\frac14\intO\frac{|\na  c|^4}{ c^3}\leq \int_{\pa\Omega}|\na \log c| ^2\kappa \left(\gamma- c\right)d\mathcal{H}^{d-1}_x+(2+d)\intO  c|\nabla^2\log c|^2,
	\end{align*}
	which implies the assertion.
\end{proof}
With the help of Lemma \ref{lem.bernstein}, the identity in Lemma \ref{lem.entropy.equality.part2} is estimated further in the next lemma.
\begin{lemma}\label{cor.entropy.equality.part2}There exists a constant $\xi>0$ such that, for all $t\in (0,T)$,
		\begin{align*}
	&\partial_t\int_{\Omega}|\nabla\sqrt{c}|^2dx 
	+\frac18 \int_{\Omega}c|\nabla^2\log c|^2dx
	+\xi \int_\Omega|\nabla \sqrt[4]{c}|^4dx
	+\int_\Omega|\nabla \sqrt{c}|^2ndx
	\\&\leq\int_{\Omega}\Delta(|\nabla\sqrt{c}|^2)dx 
	- \frac12\int_{\Omega}\nabla n\cdot\nabla cdx+
	\frac{1}{2}\int_{\Gamma}|\na \sqrt c|^2 \kappa \left(\frac{\gamma}c- 1\right)d\mathcal{H}^{d-1}_x
	+4\|c\|_{L^\infty(\Omega)}\|u\|_{H^1(\Omega)}^2.
	\end{align*}
\end{lemma}
\begin{proof}
	We apply Young's inequality and obtain
	\begin{align*}
	-2\int_\Omega \nabla\sqrt c\cdot\nabla ^Tu\nabla \sqrt{c}dx &\leq \frac14\int_\Omega \frac{|\nabla\sqrt c|^4}{c}dx+4\int_\Omega c|\nabla ^Tu|^2dx
	 \\&=\frac1{72}\int_\Omega \frac{|\nabla c|^4}{c^3}dx+4\int_\Omega c|\nabla ^Tu|^2dx
	\end{align*}
	Then by Lemma \ref{lem.bernstein}, we have
	\begin{align*}
	&-2\int_\Omega \nabla\sqrt c\cdot\nabla ^Tu\nabla \sqrt{c}dx\\
	&\leq
	 \frac{1}{16}\int_{\Gamma}|\na \log c| ^2\kappa \left(\gamma- c\right)d\mathcal{H}^{d-1}_x+\frac{2+d}{16}\intO  c|\nabla^2\log c|^2
	 +4\int_\Omega c|\nabla ^Tu|^2dx
	 \\&
	 \leq
	 \frac{1}{4}\int_{\Gamma}|\na \sqrt c|^2 \kappa \left(\frac{\gamma}c- 1\right)d\mathcal{H}^{d-1}_x+\frac{5}{16}\intO  c|\nabla^2\log c|^2
	 +4\|c\|_{L^\infty(\Omega)}\|u\|_{H^1(\Omega)}^2
	\end{align*}
	since $d\leq 3$. Using this for the estimate from Lemma \ref{lem.entropy.equality.part2} entails
			\begin{multline*}
	\partial_t\int_{\Omega}|\nabla\sqrt{c}|^2dx 
	+\frac18 \int_{\Omega}c|\nabla^2\log c|^2dx
	+\int_\Omega|\nabla \sqrt{c}|^2ndx
	\\\leq\int_{\Omega}\Delta(|\nabla\sqrt{c}|^2)dx 
	- \frac12\int_{\Omega}\nabla n\cdot\nabla cdx+
	\frac{1}{4}\int_{\Gamma}|\na \sqrt c|^2 \kappa \left(\frac{\gamma}c- 1\right)d\mathcal{H}^{d-1}_x
	+4\|c\|_{L^\infty(\Omega)}\|u\|_{H^1(\Omega)}^2
	\end{multline*}
	Finally, we use again Lemma \ref{lem.bernstein} to yield the desired assertion.
\end{proof}

Looking at Lemma \ref{cor.entropy.equality.part2}, in the case of homogeneous Neumann boundary condition $\pnu c = 0$ and convex domain, as considered in \cite{Win12}, we have $\pnu |\na \sqrt{c}|^2 \leq 0$ and therefore the term $\int_{\Omega}\Delta(|\nabla \sqrt{c}|^2)dx$ can be eliminated immediately. In the present paper, since the boundary condition is inhomogeneous and the domain is possibly not convex, we will have to deal with this term differently. Our key idea is to consider the boundary energy (see Lemma \ref{boundary_energy}). Before that we derive some useful estimates.
\begin{lemma}\label{lem.1.bd.integral}
	It holds for any smooth function $c$ satisfying $\partial_\nu c = \kappa(x)(\gamma(x) - c)$ on $\Gamma$ that
	\begin{multline}\label{eq.1.bd.integral}
\int_{\Omega}\Delta(|\nabla\sqrt c|^2)+\int_{\Gamma}|\nabla_{\Gamma}\sqrt{c}|^2
	\kappa  \left(1+\frac{\gamma}{2c}  \right)d\mathcal H^{d-1}_x
	+2\int_{\Gamma}\nabla \sqrt{c}\cdot \nabla^T\nu\nabla \sqrt{c}d\mathcal H^{d-1}_x
	\\\leq\int_{\Gamma} \partial_{\nu}|\partial_\nu\sqrt{c} |^2d\mathcal{H}^{d-1}_x
	+{
	\int_{\Gamma}\frac{4|\nabla_\Gamma \sqrt{\kappa} |^2}{\gamma} \left(\gamma-c\right)^2d\mathcal H^{d-1}_x}.
	\end{multline}
\end{lemma}
\begin{proof}
	First, we see by Gau\ss{}' theorem that
	\begin{align*}
	\frac12\int_{\Omega}\Delta(|\nabla\sqrt{c}|^2) dx &=\frac12\int_{\Gamma}\nu\cdot\nabla (|\nabla\sqrt{c}|^2) d\mathcal{H}^{d-1}_x
	\end{align*}
	Notice that
	\begin{align*}
	\frac12\nu\cdot\nabla(|\nabla \sqrt{c}|^2)
	&=\nu\cdot \nabla^2 \sqrt{c}\nabla\sqrt{c}
	\end{align*}
	Let $x\in \Gamma$ and $\vartheta:\mathbb R\to \mathbb R^d$ be differentiable such that $\vartheta(0)=x$ and $\vartheta'(0)=\nabla\sqrt{c(x)}$. Then
	\begin{align*}
	\nu(x)\cdot &\nabla^2 \sqrt{c(x)}\sqrt{c(x)}
	=
	\nu(x)\cdot \nabla^T(\nabla \sqrt{c})\circ\vartheta(0)\vartheta'(0)
	\\&=\nu\circ\vartheta(x)\cdot\frac{d}{ds}\bigg|_{s=0}(\nabla \sqrt{c})\circ\vartheta(s)
	\\&=\frac{d}{ds}\bigg|_{s=0}(\nu\cdot\nabla \sqrt{c})\circ\vartheta(s)
	-\vartheta'(0)\cdot \nabla^T\nu(x)\nabla \sqrt{c(x)}
	\\&=\vartheta'(0)\cdot \nabla (\partial_{\nu}\sqrt{c}(x))-\vartheta'(0)\cdot \nabla^T\nu(x)\vartheta'(0)
	\end{align*}
	Now, let $\tau=(1-\nu\otimes\nu)\vartheta'(0)=(1-\nu\otimes\nu)\nabla\sqrt{c} $ be the projection of $\vartheta'(0)$ onto $T_x\Gamma$. It holds (at $x$)
	\begin{align*}
	\vartheta'(0)\cdot \nabla (\partial_{\nu}\sqrt{c})
	&=\partial_\tau( \partial_{\nu}\sqrt{c})+\nu \cdot  \nabla\sqrt{c} \nu\cdot \nabla ( \partial_{\nu}\sqrt{c})
	\\
	&=\partial_\tau( \partial_{\nu}\sqrt{c})+  \partial_\nu\sqrt{c} \partial_{\nu} ( \partial_{\nu}\sqrt{c})
	\\
	&=\partial_\tau( \partial_{\nu}\sqrt{c})
	+ \partial_\nu\sqrt{c} \partial_{\nu}^2\sqrt{c}
	\end{align*}
	Using the fact that $\tau\in T_x\Gamma$ and the boundary condition for $\sqrt{c}$ in \eqref{boundary_sqrt_c}, we obtain
	\begin{align*}
	\partial_\tau ( \partial_\nu \sqrt{c}  )
	&=\partial_\tau  \left(\frac{\kappa }{2 } \left(\frac{\gamma}{\sqrt{c}} -\sqrt{c}  \right) \right)
	\\&=\partial_\tau  \kappa \left(\frac{1}{2} \left(\frac{\gamma}{\sqrt{c}} -\sqrt{c}  \right) \right)
	-\partial_\tau\sqrt{c} \frac{\kappa }{2 } \left(1+\frac{\gamma}{c}  \right)
	\\&=\partial_\tau \log \kappa \partial_\nu \sqrt{c} 
	-\partial_\tau\sqrt{c} \frac{\kappa }{2 } \left(1+\frac{\gamma}{c}  \right).
	\end{align*}
	Thus, by inserting $\tau=(1-\nu\otimes\nu)\nabla\sqrt{c} $, this entails
	\begin{align*}
	\partial_\tau \partial_\nu \sqrt{c}  
	&=
	\sqrt{c_\infty}\nabla\sqrt{c} \cdot\nabla\log\kappa\partial_\nu \sqrt{c} 
	-\partial_\nu\sqrt{c} \partial_\nu\log \kappa\partial_\nu \sqrt{c} 
	\\&\qquad
	-\nabla\sqrt{c} \cdot\nabla\sqrt{c} \frac{\kappa}{2} \left(1+\frac{\gamma}{c}  \right)
	+\partial_\nu\sqrt{c} \partial_\nu\sqrt{c} \frac{\kappa}{2} \left(1+\frac{\gamma}{c}  \right)
	\\&=
	\nabla\sqrt{c} \cdot\nabla\log\kappa\partial_\nu \sqrt{c} 
	- |\partial_\nu\sqrt{c}  |^2\partial_\nu\log\kappa
	\\&\qquad
	- (\underbrace{ |\nabla\sqrt{c}  |^2- |\partial_\nu\sqrt{c}  |^2 }_{=|\nabla_{\Gamma}\sqrt{c}|^2})\frac{\kappa}{2} \left(1+\frac{\gamma}{c}  \right).
	\end{align*}
	Thus, combining these calculations yields 
	\begin{align*}
	\frac12\nu\cdot\nabla(|\nabla \sqrt{c}|^2)
	&=\vartheta'(0)\cdot \nabla (\partial_{\nu}\sqrt{c})-\vartheta'(0)\cdot \nabla^T\nu\vartheta'(0)
	\\&=
	\partial_\tau\partial_{\nu}\sqrt{c}+  \frac12\partial_{\nu}|\partial_\nu\sqrt{c} |^2
-\vartheta'(0)\cdot \nabla^T\nu\vartheta'(0)
\\&=
	\nabla\sqrt{c} \cdot\nabla\log\kappa\partial_\nu \sqrt{c} 
- |\partial_\nu\sqrt{c}  |^2\partial_\nu\log\kappa-|\nabla_{\Gamma}\sqrt{c}|^2\frac{\kappa}{2} \left(1+\frac{\gamma}{c}  \right)
\\&\qquad +  \frac12\partial_{\nu}|\partial_\nu\sqrt{c} |^2-\nabla \sqrt{c}\cdot \nabla^T\nu\nabla \sqrt{c}
\\&=
\nabla_{\Gamma}\sqrt{c} \cdot\nabla_\Gamma\log\kappa\partial_\nu \sqrt{c} 
-|\nabla_{\Gamma}\sqrt{c}|^2\frac{\kappa}{2} \left(1+\frac{\gamma}{c}  \right)
\\&\qquad +  \frac12\partial_{\nu}|\partial_\nu\sqrt{c} |^2-\nabla \sqrt{c}\cdot \nabla^T\nu\nabla \sqrt{c}
	\end{align*}
	Using
	\begin{align*}
	\nabla_{\Gamma}\sqrt{c} \cdot\nabla_\Gamma\log \kappa \partial_\nu \sqrt{c} 
	=\frac{\nabla_{\Gamma}c}{2\sqrt c} \cdot\nabla_\Gamma\log \kappa  \frac{\kappa }{2}\left(\frac{\gamma}{\sqrt{c}}-\sqrt{c}\right)
	=\frac{\nabla_{\Gamma}c}{4 c} \cdot\nabla_\Gamma \kappa  \left(\gamma-c\right),
	\end{align*}
	this shows
		\begin{multline*}
	\int_{\Omega}\Delta(|\nabla\sqrt c|^2)+\int_{\Gamma}|\nabla_{\Gamma}\sqrt{c}|^2
	\kappa  \left(1+\frac{\gamma}{c}  \right)d\mathcal H^{d-1}_x
	+2\int_{\Gamma}\nabla \sqrt{c}\cdot \nabla^T\nu\nabla \sqrt{c}d\mathcal H^{d-1}_x
	\\=\int_{\Gamma} \partial_{\nu}|\partial_\nu\sqrt{c} |^2d\mathcal{H}^{d-1}_x
	+
	\int_{\Gamma}\frac{\nabla_{\Gamma}c}{2 c} \cdot\nabla_\Gamma \kappa  \left(\gamma-c\right)d\mathcal H^{d-1}_x.
	\end{multline*}
	Finally, we use Young's inequality to see that
	\begin{align*}
	\int_{\Gamma}\frac{\nabla_{\Gamma}c}{2 c} \cdot\nabla_\Gamma \kappa  \left(\gamma-c\right)d\mathcal H^{d-1}_x
	&\leq 
		\int_{\Gamma}\underbrace{\frac{|\nabla_{\Gamma}c|^2}{8 c^2} \gamma}_{\frac12 |\nabla_{\Gamma}\sqrt c|^2\frac{\gamma}{c}} \kappa  d\mathcal H^{d-1}_x
		+
			\int_{\Gamma}\frac{4|\nabla_\Gamma \sqrt{\kappa}|^2}{\gamma}\left(\gamma-c\right)^2 d\mathcal H^{d-1}_x,
 	\end{align*}
	which finishes the proof.
\end{proof}
To control the first term on the right hand side of \eqref{eq.1.bd.integral}, our key idea is to introduce a ``boundary energy" of the form $\int_{\Gamma}\kappa s^\infty(\gamma|c)\dH$ (where $s^\infty$ defined in \eqref{entropy}), whose time derivative produces the first term on the right-hand side of \eqref{eq.1.bd.integral} with an opposite sign (see the last term on the right-hand side of \eqref{p2}).
\begin{lemma}\label{boundary_energy}
	For all $t\in (0,T)$, it holds
	\begin{equation}\label{p2}
	\begin{aligned}
	\frac{d}{dt}\int_{\Gamma}\kappa s^\infty(\gamma|c)\dH&+
	\frac34\int_{\Gamma}\gamma|\nabla_{\Gamma}c|^2\frac{\kappa }{c^2}d\mathcal{H}^{d-1}_x
	+
	4\int_{\Gamma}\nabla\cdot\nu |\partial_{\nu}\sqrt c|^2d\mathcal{H}^{d-1}_x
	\\&+
	2\int_{\Gamma}(\partial_{\nu}\sqrt c)^2c\kappa \left(\gamma-{c}\right)d\mathcal{H}^{d-1}_x
	\\
	&\leq
{8}
\int_{\Gamma}\frac{ |\nabla_{\Gamma}\sqrt{\kappa} |^2}{ \gamma}\left(c-\gamma\right)^2d\mathcal{H}^{d-1}_x
+{
	2\int_{\Gamma}\frac{ |\nabla_{\Gamma}\gamma|^2}{\gamma}\kappa d\mathcal{H}^{d-1}_x}	
\\	&+
	\int_{\Gamma}n\kappa \left(\gamma-{c}\right)d\mathcal{H}^{d-1}_x
	-
	2\int_{\Gamma}\partial_{\nu}|\partial_{\nu}\sqrt c|^2d\mathcal{H}^{d-1}_x
.
	\end{aligned}
	\end{equation}	
\end{lemma}
\begin{proof}
	Thanks to the relation between the Laplace operator and Laplace-Beltrami operator $
	\Delta_{\Gamma}=\Delta-\nabla\cdot\nu \partial_\nu-\partial_\nu^2$ on $\Gamma$,
	and the fact that $u = 0$ on $\Gamma$, we have
	\begin{align}
	\label{eq4}
	\partial_t{c}&=
	\Delta_{\Gamma}{c}+\partial_\nu^2{c}+ \nabla\cdot\nu \partial_\nu {c}-{ c}n \quad \text{ on } \Gamma.
	\end{align}
	We can therefore compute
	\begin{align*}
	\frac{d}{dt}\int_{\Gamma}\kappa s^\infty(\gamma|c)\dH&=\int_{\Gamma}\partial_t c\kappa \left(1-\frac\gamma{c}\right)d\mathcal{H}^{d-1}_x
	\\&=
	\int_{\Gamma}\Delta_{\Gamma }c\kappa \left(1-\frac\gamma{c}\right)d\mathcal{H}^{d-1}_x
	+
	\int_{\Gamma}\nabla\cdot \nu\partial_{\nu}c\kappa \left(1-\frac\gamma{c}\right)d\mathcal{H}^{d-1}_x
	\\&\qquad+
	\int_{\Gamma}\partial_\nu^2 c\kappa \left(1-\frac\gamma{c}\right)d\mathcal{H}^{d-1}_x-
	\int_{\Gamma}cn\kappa \left(1-\frac\gamma{c}\right)d\mathcal{H}^{d-1}_x
	\end{align*}
	Using integration by parts, we have
	\begin{align*}
	\int_{\Gamma}\Delta_{\Gamma }c\kappa \left(1-\frac\gamma{c}\right)d\mathcal{H}^{d-1}_x&=
	-
	\int_{\Gamma}\nabla_{\Gamma}c\cdot \nabla_{\Gamma}\kappa \left(1-\frac\gamma{c}\right)d\mathcal{H}^{d-1}_x-
	\int_{\Gamma}\gamma|\nabla_{\Gamma}c|^2\frac{\kappa }{c^2}d\mathcal{H}^{d-1}_x
	\\&\qquad{+
	\int_{\Gamma}\frac{\kappa }{c}\nabla_{\Gamma}c\cdot \nabla_{\Gamma}\gamma d\mathcal{H}^{d-1}_x}.
	\end{align*}	We calculate
	\begin{align*}
	\partial_\nu^2 c\kappa \left(1-\frac\gamma{c}\right)&=4\sqrt c \partial_{\nu}^2\sqrt c\frac{\kappa }2\left(1-\frac\gamma{c}\right)+2(\partial_{\nu}\sqrt c)^2\kappa \left(1-\frac\gamma{c}\right)
	\\&=
	4\sqrt c \partial_{\nu}^2\sqrt c\frac{\kappa }2\left(1-\frac\gamma{c}\right)+2(\partial_{\nu}\sqrt c)^2\kappa \left(1-\frac\gamma{c}\right)
	\\&=
	-4\partial_{\nu}\sqrt c \partial_{\nu}^2\sqrt c+2(\partial_{\nu}\sqrt c)^2\kappa \left(1-\frac\gamma{c}\right)
	\\&=
	-2\partial_{\nu}|\partial_{\nu}\sqrt c|^2+2(\partial_{\nu}\sqrt c)^2\kappa \left(1-\frac\gamma{c}\right)
	\end{align*}
	using $\partial_{\nu}\sqrt{c}=\frac{\kappa }{2}\left(\frac{\gamma}{\sqrt{c}}-\sqrt{c}\right)=-\frac{\kappa }{2}\sqrt{c}\left(1-\frac{\gamma}{c}\right)$. 
	Combing these equalities yields
		\begin{align*}
	\frac{d}{dt}\int_{\Gamma}\kappa s^\infty(\gamma|c)dx&=
		-
	\int_{\Gamma}\nabla_{\Gamma}c\cdot \nabla_{\Gamma}\kappa \left(1-\frac\gamma{c}\right)d\mathcal{H}^{d-1}_x-
	\int_{\Gamma}\gamma|\nabla_{\Gamma}c|^2\frac{\kappa }{c^2}d\mathcal{H}^{d-1}_x
	\\&\qquad
	+
	\int_{\Gamma}\nabla\cdot \nu\partial_{\nu}c\kappa \left(1-\frac\gamma{c}\right)d\mathcal{H}^{d-1}_x-
	2\int_{\Gamma}\partial_{\nu}|\partial_{\nu}\sqrt c|^2d\mathcal{H}^{d-1}_x
	\\&\qquad+
	2\int_{\Gamma}(\partial_{\nu}\sqrt c)^2\kappa \left(1-\frac\gamma{c}\right)d\mathcal{H}^{d-1}_x
	-
	\int_{\Gamma}cn\kappa \left(1-\frac\gamma{c}\right)d\mathcal{H}^{d-1}_x\\
	&\qquad + \int_{\Gamma}\frac{\kappa }{c}\nabla_{\Gamma}c\cdot \nabla_{\Gamma}\gamma d\mathcal{H}^{d-1}_x.
	\end{align*}
	It remains to estimate the first, the third and the last terms on the right-hand side. Finally Young's inequality implies 
	\begin{align*}-
	\int_{\Gamma}\nabla_{\Gamma}c\cdot \nabla_{\Gamma}\kappa \left(1-\frac\gamma{c}\right)d\mathcal{H}^{d-1}_x
&\leq {\frac18}
	\int_{\Gamma}\gamma|\nabla_{\Gamma}c|^2\frac{\kappa }{c^2}d\mathcal{H}^{d-1}_x
	+{8}
		\int_{\Gamma}\frac{ |\nabla_{\Gamma}\sqrt{\kappa} |^2}{\gamma}\left(c-\gamma\right)^2d\mathcal{H}^{d-1}_x
	\end{align*}
	{and}
	\begin{align*}
	{
		\int_{\Gamma}\frac{\kappa }{c}\nabla_{\Gamma}c\cdot \nabla_{\Gamma}\gamma d\mathcal{H}^{d-1}_x}&{\leq
	\frac18
	\int_{\Gamma}\gamma|\nabla_{\Gamma}c|^2\frac{\kappa }{c^2}d\mathcal{H}^{d-1}_x+2
	\int_{\Gamma}\frac{ |\nabla_{\Gamma}\gamma|^2}{\gamma}\kappa d\mathcal{H}^{d-1}_x.}
	\end{align*}
	These two estimates and the identity
\begin{align*}
\int_{\Gamma}\nabla\cdot \nu\partial_{\nu}c\kappa \left(1-\frac\gamma{c}\right)d\mathcal{H}^{d-1}_x
&=	
\int_{\Gamma}\nabla\cdot\nu \frac{\partial_{\nu}c}{c}\kappa \left(c-\gamma\right)d\mathcal{H}^{d-1}_x
\\&=-
\int_{\Gamma}\nabla\cdot\nu \frac{|\partial_{\nu}c|^2}{c}d\mathcal{H}^{d-1}_x
\\&=-
4\int_{\Gamma}\nabla\cdot\nu |\partial_{\nu}\sqrt c|^2d\mathcal{H}^{d-1}_x
\end{align*}
imply the assertion of Lemma \ref{boundary_energy}.
	\end{proof}
	By combining Lemmas \ref{lem.1.bd.integral} and \ref{boundary_energy}, we have the following lemma.
\begin{lemma}\label{lem.entropy.equality.part3}
	It holds for all $t\in (0,T)$ that
	\begin{multline*}
	\frac{d}{dt}\int_{\Gamma}\kappa s^\infty(\gamma|c)\dH+
	2\int_{\Omega}\Delta(|\nabla\sqrt c|^2)dx+2\int_{\Gamma}|\nabla_{\Gamma}\sqrt{c}|^2
	\kappa  \left(1+2\frac{\gamma}{c}  \right)d\mathcal H^{d-1}_x
	\\	+4\int_{\Gamma}\nabla \sqrt{c}\cdot \nabla^T\nu\nabla \sqrt{c}d\mathcal H^{d-1}_x
	+	
	2\int_{\Gamma}\left(2\nabla\cdot\nu +c\kappa (\gamma-c)\right)|\partial_{\nu}\sqrt c|^2d\mathcal{H}^{d-1}_x
	\\
	\leq 
	{16
	\int_{\Gamma}\frac{ |\nabla_{\Gamma}\sqrt\kappa |^2}{ \gamma}\left(c-\gamma\right)^2d\mathcal{H}^{d-1}_x
	+
		2\int_{\Gamma}\frac{ |\nabla_{\Gamma}\gamma|^2}{\gamma}\kappa d\mathcal{H}^{d-1}_x}	
	+
	\int_{\Gamma}n\kappa \left(\gamma-{c}\right)d\mathcal{H}^{d-1}_x.
	\end{multline*}
\end{lemma}
\begin{proof}
	We add the following two estimates
		\begin{multline}
\int_{\Omega}\Delta(|\nabla\sqrt c|^2)dx+\int_{\Gamma}|\nabla_{\Gamma}\sqrt{c}|^2
\kappa  \left(1+\frac{\gamma}{2c}  \right)d\mathcal H^{d-1}_x
+2\int_{\Gamma}\nabla \sqrt{c}\cdot \nabla^T\nu\nabla \sqrt{c}d\mathcal H^{d-1}_x
\\\leq\int_{\Gamma} \partial_{\nu}|\partial_\nu\sqrt{c} |^2d\mathcal{H}^{d-1}_x
+
4\int_{\Gamma}\frac{|\nabla_\Gamma \sqrt\kappa |^2}{ \gamma} \left(\gamma-c\right)^2d\mathcal H^{d-1}_x
	\end{multline}
	and
		\begin{multline*}
	\frac12\frac{d}{dt}\int_{\Gamma}\kappa s^\infty(\gamma|c)dx+
	\frac32\int_{\Gamma}\gamma|\nabla_{\Gamma}\sqrt c|^2\frac{\kappa }{c}d\mathcal{H}^{d-1}_x
	\\+
	2\int_{\Gamma}\nabla\cdot\nu |\partial_{\nu}\sqrt c|^2d\mathcal{H}^{d-1}_x
	+
	\int_{\Gamma}(\partial_{\nu}\sqrt c)^2c\kappa \left(\gamma-{c}\right)d\mathcal{H}^{d-1}_x
	\\\leq 
	4\int_{\Gamma}\frac{ |\nabla_{\Gamma}\sqrt\kappa |^2}{ \gamma}\left(c-\gamma\right)^2d\mathcal{H}^{d-1}_x
	+
	{
		\int_{\Gamma}\frac{ |\nabla_{\Gamma}\gamma|^2}{\gamma}\kappa d\mathcal{H}^{d-1}_x}	
\\	+
	\frac12\int_{\Gamma}n\kappa \left(\gamma-{c}\right)d\mathcal{H}^{d-1}_x
	-
	\int_{\Gamma}\partial_{\nu}|\partial_{\nu}\sqrt c|^2d\mathcal{H}^{d-1}_x
	\end{multline*}
	to obtain
	\begin{multline*}
	\frac12\frac{d}{dt}\int_{\Gamma}\kappa s^\infty(\gamma|c)dx+
		\int_{\Omega}\Delta(|\nabla\sqrt c|^2)+\int_{\Gamma}|\nabla_{\Gamma}\sqrt{c}|^2
	\kappa  \left(1+2\frac{\gamma}{c}  \right)d\mathcal H^{d-1}_x
	\\	+2\int_{\Gamma}\nabla \sqrt{c}\cdot \nabla^T\nu\nabla \sqrt{c}d\mathcal H^{d-1}_x
	+	
	\int_{\Gamma}\left(2\nabla\cdot\nu +c\kappa (\gamma-c)\right)|\partial_{\nu}\sqrt c|^2d\mathcal{H}^{d-1}_x
	\\
	\leq 
	8
	\int_{\Gamma}\frac{ |\nabla_{\Gamma}\sqrt\kappa |^2}{\gamma}\left(c-\gamma\right)^2d\mathcal{H}^{d-1}_x
	+
	{
		\int_{\Gamma}\frac{ |\nabla_{\Gamma}\gamma|^2}{\gamma}\kappa d\mathcal{H}^{d-1}_x}	
	+
	\frac12\int_{\Gamma}n\kappa \left(\gamma-{c}\right)d\mathcal{H}^{d-1}_x.\qedhere
	\end{multline*}
\end{proof}
From Lemmas \ref{lem.ee.n}, \ref{cor.entropy.equality.part2}, \ref{lem.entropy.equality.part3} we have the preliminary energy estimates for $n$ and $c$.
\begin{lemma}\label{energy}
	For all $t\in (0,T)$ we have
	\begin{equation}\label{energy_1}
	\begin{aligned}
		\frac d{dt}\left(\int_{\Omega}s(n)dx+	2\int_{\Omega}|\nabla\sqrt{c}|^2 dx+\int_{\Gamma}\kappa s^\infty(\gamma|c)\dH\right)
		\\+4\int_{\Omega}\left|\nabla\sqrt{n}\right|^2 dx 
		+\epsilon\int_{\Omega}n(n^2-1)\log ndx
		\\	+\frac14 \int_{\Omega}c|\nabla^2\log c|^2dx
		+\xi \int_\Omega|\nabla \sqrt[4]{c}|^4dx
		+2\int_\Omega|\nabla \sqrt{c}|^2ndx
		\\+2\int_{\Gamma}|\nabla_{\Gamma}\sqrt{c}|^2
		\kappa  \left(1+\frac{\gamma}{c}  \right)d\mathcal H^{d-1}_x
			+4\int_{\Gamma}\nabla \sqrt{c}\cdot \nabla^T\nu\nabla \sqrt{c}d\mathcal H^{d-1}_x
		\\+	
		2\int_{\Gamma}\left(2\nabla\cdot\nu +c\kappa (\gamma-c)\right)|\partial_{\nu}\sqrt c|^2d\mathcal{H}^{d-1}_x
	\\
	\leq 
	{16
		\int_{\Gamma}\frac{ |\nabla_{\Gamma}\sqrt\kappa |^2}{\gamma}\left(c-\gamma\right)^2d\mathcal{H}^{d-1}_x
		+
		2\int_{\Gamma}\frac{ |\nabla_{\Gamma}\gamma|^2}{\gamma}\kappa d\mathcal{H}^{d-1}_x}	
	\\
		+
	\int_{\Gamma}n\kappa \left(\gamma-{c}\right)d\mathcal{H}^{d-1}_x
	+8\|c\|_{L^\infty(\Omega)}\|u\|_{H^1(\Omega)}^2.
	\end{aligned}
\end{equation}
\end{lemma}
\begin{proof}
	We recall Lemma \ref{lem.ee.n}
	\begin{equation*}
	\frac d{dt}\int_{\Omega}s(n)dx+4\int_{\Omega}\left|\nabla\sqrt{n}\right|^2 dx+\epsilon\int_{\Omega}n\log n(1-n^2)dx
=\int_{\Omega}\nabla c\cdot \nabla  ndx
	\end{equation*}
and	Lemma \ref{cor.entropy.equality.part2}
	\begin{multline*}
2\partial_t\int_{\Omega}|\nabla\sqrt{c}|^2dx 
+\frac14 \int_{\Omega}|\nabla^2\log c|^2cdx
+\xi \int_\Omega|\nabla \sqrt[4]{c}|^4dx
+2\int_\Omega|\nabla \sqrt{c}|^2ndx
\\\leq2\int_{\Omega}\Delta(|\nabla\sqrt{c}|^2)dx 
- \int_{\Omega}\nabla n\cdot\nabla cdx+
\int_{\Gamma}|\na \sqrt c|^2 \kappa \left(\frac{\gamma}c- 1\right)d\mathcal{H}^{d-1}_x
+8\int_\Omega c|\nabla ^Tu|^2dx
\end{multline*}
for some $\xi>0$ and Lemma \ref{lem.entropy.equality.part3}
\begin{multline*}
\frac{d}{dt}\int_{\Gamma}\kappa s^\infty(\gamma|c)dx+
2\int_{\Omega}\Delta(|\nabla\sqrt c|^2)dx+2\int_{\Gamma}|\nabla_{\Gamma}\sqrt{c}|^2
\kappa  \left(1+2\frac{\gamma}{c}  \right)d\mathcal H^{d-1}_x
\\	+4\int_{\Gamma}\nabla \sqrt{c}\cdot \nabla^T\nu\nabla \sqrt{c}d\mathcal H^{d-1}_x
+	
2\int_{\Gamma}\left(2\nabla\cdot\nu +c\kappa (\gamma-c)\right)|\partial_{\nu}\sqrt c|^2d\mathcal{H}^{d-1}_x
\\
\leq 
{16
	\int_{\Gamma}\frac{ |\nabla_{\Gamma}\sqrt\kappa |^2}{\gamma}\left(c-\gamma\right)^2d\mathcal{H}^{d-1}_x
	+
	2\int_{\Gamma}\frac{ |\nabla_{\Gamma}\gamma|^2}{\gamma}\kappa d\mathcal{H}^{d-1}_x}	
+
\int_{\Gamma}n\kappa \left(\gamma-{c}\right)d\mathcal{H}^{d-1}_x.
\end{multline*}
By adding these three relations, we obtain the assertion of Lemma \ref{energy}. 
\end{proof}
	The form of energy estimate in Lemma \ref{energy} is particularly suited for convex domains as then $\nabla\cdot \nu\geq0$ and $\nabla^T\nu$ is positive semi-definite on $\Gamma$. Therefore, the terms involving $\nabla\cdot\nu$ and $\nabla^T\nu$ on the l.h.s.\ in the energy estimate can be neglected as they are non-negative. 
	
	\medskip
	When the domain is not convex, we show in the next lemma that these terms can be controled using the higher order terms $\int_{\Omega}c|\na^2\log c|^2dx$ and $\int_{\Omega}|\na\sqrt[4]{c}|^4dx$.

\begin{lemma}
	\label{lem:energy2}
	There exists a $\lambda>0$ and a $C>0$ depending on $\kappa,\gamma,\|c\|_{L^\infty}$ and the curvature of $\Gamma$ such that
	\begin{align*}
	\frac d{dt}\left(\int_{\Omega}s(n)dx+	2\int_{\Omega}|\nabla\sqrt{c}|^2 dx+\int_{\Gamma}\kappa s^\infty(\gamma|c)\dH\right)
	\\+4\int_{\Omega}\left|\nabla\sqrt{n}\right|^2 dx 
	+\epsilon\int_{\Omega}n(n^2-1)\log ndx
	\\	+\lambda\int_{\Omega}|\nabla^2\sqrt{c}|^2dx+\lambda \int_\Omega|\nabla \sqrt[4]{c}|^4dx
	+2\int_\Omega|\nabla \sqrt{c}|^2ndx
	\\+2\int_{\Gamma}|\nabla_{\Gamma}\sqrt{c}|^2
	\kappa  \left(1+\frac{\gamma}{c}  \right)d\mathcal H^{d-1}_x	
	\\
	\leq 
	C\left(1+\int_{\Omega}|\nabla\sqrt{c}|^2dx\right)+8\|c\|_{L^\infty(\Omega)}\|u\|_{H^1(\Omega)}^2.
	\end{align*}
\end{lemma}
\begin{proof}
	From Lemma \ref{energy}, we need to control the tenth and eleventh terms on the left-hand side, and the first, second and third terms on the right-hand side of \eqref{energy_1}.

	The first step is to show that we can estimate the $H^2(\Omega)$ norm of $\sqrt c$. We need to have a closer look at the integral involving $	|\nabla^2\log c|^2c$. 
	Using the chain rule and $|a+b|^2\geq \frac 12|a|^2- |b|^2$, we have
	\begin{align*}
	|\nabla^2\log c|^2c
	&=
	\left|\sqrt{c}\nabla\left(\frac{\nabla {c}}{{c}}\right)\right|^2
	=
	4\left|\sqrt{c}\nabla\left(\frac{\nabla \sqrt{c}}{\sqrt{c}}\right)\right|^2
	\\&=4
	\left|\nabla^2\sqrt{c}-\frac{1}{\sqrt{c}}\nabla\sqrt{c}\otimes\nabla\sqrt{c}\right|^2
	\\&
	\geq 2|\na^2\sqrt{c}|^2-4\frac{|\nabla\sqrt{c}|^4}{c}=2|\na^2\sqrt{c}|^2-4\cdot{2^4}{|\nabla\sqrt[4]{c}|^4}.
	\end{align*}
	This directly implies
	\begin{align*}
	\int_{\Omega}|\nabla^2\sqrt{c}|^2 dx\leq \frac12 \int_{\Omega} |\nabla^2\log c|^2c\,dx + 32\int_{\Omega}{|\nabla\sqrt[4]{c}|^4}dx.
	\end{align*}
	We use the fact that (see e.g. \cite[Theorem 1.5.1.10]{Gri85} for any $\theta>0$, there exists $C_\theta > 0$ such that
	\begin{equation}\label{trace}
	\int_{\Gamma}|f|^2\dH \leq \theta\int_{\Omega}|\nabla f|^2dx + C_{\theta}\int_{\Omega}|f|^2dx.
	\end{equation}
	
	The tenth term on the left hand side of \eqref{energy_1} can be estimated as, for any $\theta>0$,
	\begin{align*}
	\left|4\int_{\Gamma}\nabla \sqrt{c}\cdot \nabla^T\nu\nabla \sqrt{c}d\mathcal H^{d-1}_x\right|
	&\leq 4\|\nabla^T\nu\|_{L^\infty(\Gamma)}\|\nabla\sqrt c\|_{L^2(\Gamma)}^2
	\\&=\theta \int_{\Omega}|\nabla^2\sqrt{c}| dx +C_{\theta,\Gamma}\int_{\Omega}|\nabla\sqrt{c}|^2 dx
	\end{align*}
	for some $C_{\theta,\Gamma}>0$ depending on $\theta$. Likewise for the first part of the eleventh term, for any $\theta > 0$,
	\begin{align*}
	\left|4\int_{\Gamma}\nabla\cdot\nu|\partial_{\nu}\sqrt c|^2d\mathcal{H}^{d-1}_x\right|
	&\leq 
	\theta \int_{\Omega}|\nabla^2\sqrt{c}| dx +C_{\theta,\Gamma}\int_{\Omega}|\nabla\sqrt{c}|^2 dx,
	\end{align*}
	using that $\kappa$ and $\gamma$ are uniformly bounded. The second part of the eleventh term on the left-hand side of \eqref{energy_1} is estimated as follows
	\begin{align*}
		&\left|2\int_{\Gamma}c\kappa(\gamma - c)|\pnu \sqrt{c}|^2\dH\right|\\
		&\leq 2\|\kappa\|_{L^\infty(\Gamma)}\|\gamma\|_{L^\infty(\Gamma)}\int_{\Gamma}c|\pnu \sqrt c|^2\dH + 2\|\kappa\|_{L^{\infty}(\Gamma)}\int_{\Gamma}c^2|\pnu \sqrt c|^2\dH\\
		&\leq C\int_{\Gamma}|\pnu c|^2\dH + C\int_{\Gamma}c|\pnu c|^2\dH\\
		&= C\int_{\Gamma}\kappa^2(\gamma - c)^2\dH + C\int_{\Gamma}c\kappa^2(\gamma - c)^2\dH\\
		&\leq C\left(1 + \int_{\Gamma}|c|^3\dH\right)\\
		&\leq C\left(1 + \int_{\Omega}c^{\frac 32}|\na\sqrt c|^2dx + \|c\|_{L^3(\Omega)}^3\right)\\
		&\leq C\left(1+\int_{\Omega}|\na\sqrt c|^2dx\right).
	\end{align*}
	The second term on the right-hand side of \eqref{energy_1} is bounded by
	\begin{equation*}
		\left|2\int_{\Gamma}\frac{|\na_\Gamma \gamma|^2}{\gamma}\kappa \dH\right| \leq \|\kappa\|_{L^\infty(\Gamma)}\|\na\sqrt{\gamma}\|_{L^2(\Gamma)}^2 \leq C.
	\end{equation*}	
	The third term on the right-hand side of \eqref{energy_1} is estimated as
	\begin{align*}
	\int_{\Gamma}n\kappa(\gamma - c)\dH &\leq \|\kappa\|_{L^\infty(\Gamma)}\|\gamma\|_{L^\infty(\Gamma)}\int_{\Gamma}|n|\dH\\
	&\leq C\int_{\Gamma}|\sqrt{n}|^2\dH\\
	&\leq 2\int_{\Omega}|\na\sqrt n|^2dx + C\int_{\Omega}|\sqrt n|^2dx\\
	&\leq 2\int_{\Omega}|\na\sqrt n|^2dx + C
	\end{align*}
	thanks to the Trace inequality \eqref{trace} and the fact that $\|n\|_{L^1(\Omega)} \leq C$ in Lemma \ref{L1Linf}. We estimate the first term on the right hand side of \eqref{energy_1} as
	\begin{align}\label{b0}
		\left|16\int_{\Gamma}\frac{|\na_\Gamma \sqrt\kappa|^2}{\gamma}(c-\gamma)^2\dH\right| &\leq 32\|\na\sqrt{\kappa}\|_{L^2(\Gamma)}^2\|1/\gamma\|_{L^\infty(\Gamma)}\left(\|\gamma\|_{L^\infty(\Gamma)}^2 + \|c\|_{L^\infty(\Gamma)}^2\right).
	\end{align}
	We now show that 
	\begin{equation}\label{b1}
		\|c\|_{L^\infty(\Gamma)}^2 \leq C\left(1+ \|\nabla \sqrt{c}\|_{L^2(\Omega)}^2\right).
	\end{equation}
	Indeed, for any $p>2$, we have
	\begin{align*}
		\|c\|_{L^p(\Gamma)}^p &=\int_{\Gamma}\left(|\sqrt{c}|^p\right)^2\dH\\
		&\leq C\left(\frac{p^2}{4}\int_{\Omega}|c|^{p-1}|\nabla \sqrt c|^2dx + \int_{\Omega}|c|^pdx \right)\\
		&\leq C\left(p^2\|c\|_{L^\infty(\Omega)}^{p-1}\|\nabla \sqrt c\|_{L^2(\Omega)}^2 + \|c\|_{L^p(\Omega)}^p\right)\\
		&\leq C\left(\|\nabla \sqrt{c}\|_{L^2(\Omega)}^p + p^{\frac{2p}{p-2}}\|c\|_{L^\infty(\Omega)}^{\frac{p(p-1)}{p-2}} + \|c\|_{L^p(\Omega)}^p \right).
	\end{align*}
	By taking root with order $p$ of both sides and letting $p\to\infty$, we get
	\begin{equation*}
		\|c\|_{L^\infty(\Gamma)} \leq C\left(\|\nabla \sqrt{c}\|_{L^2(\Omega)} + \|c\|_{L^\infty(\Omega)}\right),
	\end{equation*}
	hence \eqref{b1} thanks to the boundedness of $\|c\|_{L^\infty(\Omega)}$. Inserting \eqref{b1} into \eqref{b0}, we have controlled the first term on the right-hand side of \eqref{energy_1}, and thus completes the proof of Lemma \ref{lem:energy2}.
\end{proof}

\begin{lemma}\label{energy-NS}
	For any $\delta > 0$, there exists $C_\delta$ depending on $\delta$ and  $\|\varphi\|_{W^{1,\rho}(\Omega)}$ such that
	\begin{equation}\label{a}
		\frac{d}{dt}\int_{\Omega}|u|^2dx + C(\mu)\|u\|_{H^1(\Omega)}^2 \leq C + \delta\int_{\Omega}|\na \sqrt n|^2dx.
	\end{equation}
\end{lemma}
\begin{proof}
	From the well-known energy estimate for the approximate Navier-Stokes equations in \eqref{Smod}
	and the Poincar\'e inequality $\|\nabla u\|_{L^2(\Omega)} \geq C\|u\|_{L^2(\Omega)}$, we have
	\begin{equation}\label{aa}
		\frac{d}{dt}\|u\|_{L^2(\Omega)}^2 + C(\mu)\|u\|_{H^1(\Omega)}^2 \leq 2\left|\int_{\Omega}n\na\varphi \cdot u dx\right|.
	\end{equation}
	We now show that for any $\delta_0, \delta_1>0$,
	\begin{equation}\label{aaa}
		\left|2\int_{\Omega}n\nabla \varphi \cdot udx\right| \leq C + \delta_0\|n\|_{L^{3}(\Omega)} + \delta_1\|u\|_{H^1(\Omega)}^2.
	\end{equation}
	By H\"older's inequality and the continuous embedding $H^1(\Omega) \hookrightarrow L^6(\Omega)$ (since $d\leq 3$) we have
	\begin{equation*}
		\left|\int_{\Omega}n\nabla \varphi \cdot udx\right| \leq \|u\nabla \varphi \|_{L^{\frac 65}(\Omega)}\|u\|_{L^6(\Omega)} \leq C\|u\nabla \varphi \|_{L^{\frac 65}(\Omega)}\|u\|_{H^1(\Omega)} \leq C\|u\nabla \varphi\|_{L^{\frac 65}(\Omega)}^2 + \delta_1\|u\|_{H^1(\Omega)}^2.
	\end{equation*}
	Let $\eta = \frac{5\rho}{6} > 5$ and $\beta = \frac{\eta}{\eta - 1} < \frac 54$ (recalling $\rho>6$ in \eqref{varphi}). By H\"older's inequality again, it follows that
	\begin{align}\label{bb}
		\|n\nabla \varphi\|_{L^{\frac 65}(\Omega)}^2 
		\leq \|\na\varphi\|_{L^{\frac 65 \eta}(\Omega)}^2\|n\|_{L^{\frac 65 \beta}(\Omega)}^2.
	\end{align}
	By using the interpolation inequality with
	\begin{align}\label{bbb}
		\|n\|_{L^{\frac 65\beta}(\Omega)}^2 \leq \|n\|_{L^1(\Omega)}^{2\theta}\|n\|_{L^3(\Omega)}^{2(1-\theta)} \quad \text{ with } \quad \frac{5}{6\beta} = \frac{\theta}{1} + \frac{1-\theta}{3}.
	\end{align}
	From that $\theta = \frac{5-2\beta}{4\beta}$ and therefore
	\begin{equation*}
		2(1-\theta) = \frac{6\beta - 5}{2\beta} < 1
	\end{equation*}
	since $\beta < \frac 54$. From \eqref{bb}, \eqref{bbb} and $\|n\|_{L^1(\Omega)}\leq C$ we obtain
	\begin{equation*}
		C\|n\na\varphi\|_{L^{\frac 65}(\Omega)}^2 \leq C\|\varphi\|_{W^{1,\rho}(\Omega)}^2C^{2\theta}\|n\|_{L^3(\Omega)}^{2(1-\theta)} \leq C + \delta_0\|n\|_{L^3(\Omega)}
	\end{equation*}
	where we used Young's inequality at the last step, due to $2(1-\theta)<1$. To obtain \eqref{a} from \eqref{aa} and \eqref{aaa}, it remains to show that
	\begin{equation*}
		\|n\|_{L^3(\Omega)} \leq C\left(1+\int_{\Omega}|\na \sqrt n|^2dx\right).
	\end{equation*}
	Indeed, thanks to the continuous three dimensional embedding $H^1(\Omega)\hookrightarrow L^6(\Omega)$, we have
	\begin{equation*}
		\|n\|_{L^3(\Omega)} = \|\sqrt n\|_{L^6(\Omega)}^2 \leq C\left(\int_{\Omega}|\na \sqrt n|^2dx + \int_{\Omega}|\sqrt n|^2dx\right) \leq C\left(1 + \int_{\Omega}|\na\sqrt n|^2dx\right)
	\end{equation*}
	thanks $\|n\|_{L^1(\Omega)} \leq C$ from Lemma \ref{L1Linf}.
\end{proof}

\begin{lemma}\label{a-priori}
	The following {\it a priori} estimtates hold uniformly in $\epsilon\geq 0$ and $m\in \mathbb N$,
	\begin{equation*}
		\begin{aligned}
		&\sup_{t\in (0,T)}\left(\int_{\Omega}n(t)\log n(t) dx + \|\na\sqrt{c}(t)\|_{L^2(\Omega)}^2 + \|u(t)\|_{L^2(\Omega)}^2\right) \leq C_T,\\
		&\int_0^T\int_{\Omega}|\na\sqrt{n}|^2dxdt + \int_0^T\int_{\Omega}|\nabla u|^2dxdt \leq C_T,\\
		&\int_0^T\int_{\Omega}c|\nabla^2 \log c|^2dxdt + \int_0^T\int_{\Omega}\frac{|\nabla c|^4}{c^3}dxdt \leq C_T,\\
		&\int_0^T\int_{\Omega}|\na \sqrt{c}|^2n dxdt \leq C_T,\\
		&\epsilon \int_0^T\int_{\Omega} n(n^2-1)\log n dxdt \leq C_T,
		\end{aligned}
	\end{equation*}
	where $C_T$ is a constant depending continuously on $T>0$.
\end{lemma}
\begin{proof}
	Define
	\begin{equation}\label{def_F}
	\mathcal F(n,c,u) = \int_{\Omega}n\log n dx + 2\int_{\Omega}|\nabla \sqrt{c}|^2dx + \int_{\Gamma}\kappa s^\infty(\gamma|c)\dH + K\int_{\Omega}|u|^2dx
	\end{equation}
	where $K$ is a large enough constant such that
	\begin{equation*}
		K\frac{C(\mu)}{2} \geq 16\|c\|_{L^\infty(\Omega)}
	\end{equation*}
	with $C(\mu)$ is in Lemma \ref{energy-NS}. It follows from Lemmas \ref{lem:energy2} and \ref{energy-NS} that
	\begin{equation}\label{k1}
	\begin{aligned}
		\mathcal F(n,c,u)(t) + C\int_s^t\left[\int_{\Omega}\Bigl(|\na\sqrt{n}|^2 + c|\na^2\log c|^2  + \frac{|\na c|^4}{c^3} + |\na \sqrt{c}|^2n\Bigr)dx + \|u\|_{H^1(\Omega)}^2\right]dr\\
		+ \epsilon \int_s^t\int_{\Omega}n(n^2-1)\log n dxdr\\
		\leq \mathcal F(n,c,u)(s) + C(t-s) + C\delta\int_s^t\int_{\Omega}|\na\sqrt n|^2dxdr +  C\int_s^t\int_{\Omega}|\na\sqrt c|^2dxdr
	\end{aligned}
	\end{equation}
	for any $\delta > 0$. In particular, by choosing $\delta$ small enough, it follows that
	\begin{equation*}
		\mathcal F(n,c,u)(t) \leq \mathcal F(n,c,u)(s) + C(t-s) + C\int_s^t\mathcal F(n,c,u)(r)dr
	\end{equation*}
	for all $0<s<t<T$, and thus for all $t\in (0,T)$,
	\begin{equation*}
		\mathcal F(n,c,u)(t) \leq C_T.
	\end{equation*}
	From this and \eqref{k1} we obtain the desired bounds in Lemma \ref{a-priori}.
	for all $t\in (0,T)$.
\end{proof}
	
\section{Global existence of solutions}\label{proofs}
\subsection{In one or two dimensions}
\begin{proof}[Proof of Theorem \ref{thm:main2D}.]
According to Proposition \ref{a-prop-local-solution-hom}, the system \eqref{C-NS}-\eqref{boundary} admits a unique classical solution on $(0,T)$ for all $T<T_{\max}$ for the maximal time $T_{\max}\in(0,\infty]$.
	We can reformulate the blow up criterion from to Proposition \ref{a-prop-local-solution-hom} in spatial dimension two to
\begin{equation}\label{c-formula-explosion-condition}
\|n(t)\|_{L^{3}(\Omega)}
+\|\na n(t)\|_{L^{3}(\Omega)}
+\|\na c(t)\|_{L^{4}(\Omega)} + \|A^\alpha u(t)\|_{L^2_\sigma(\Omega)}\rightarrow \infty \ \text{as } t\nearrow T_{\max}
\end{equation}
 if $T_{\max}<\infty$ using the Sobolev embedding $W^{1,3}(\Omega)\subset L^\infty(\Omega)$.  
%
	In order to prevent blow up, the first step is to bound $n$ in $L^p$ for all $1\leq p <\infty$. Considering the time derivative of the $L^p$ norm of $n$ yields the following result, which is based on a lemma from \cite{Win12}. It's worth to remark that this trick only works in one or two spacial dimensions.
	\begin{lemma}\cite[Lemma 3.7]{Bra17}\label{e-lem-estimate-lp-grad-n-in-2D}
		If $d\leq 2$ and $p>1$, then there exists a constant $C_p>0$ such that
		\begin{align}\label{n_Lp}
		\frac1p\frac{d}{dt} \int_\Omega n^pdx +\frac{p-1}2\int_\Omega|\na n|^2n^{p-2}dx\leq C_p\left(\int_\Omega|\na c|^4dx+1\right)\int_\Omega n^pdx
		\end{align}
		holds for all $t\in (0,T)$.
	\end{lemma}
	
	It is remarked that the proof of this Lemma does not use any information of the logistic source (as it was included in the model in \cite{Bra17}), and therefore it is applicable for \eqref{C-NS}.
	
	\medskip
From Lemma \ref{e-lem-estimate-lp-grad-n-in-2D}, it is crucial to get
\begin{equation}\label{crucial}
\int_0^T\int_{\Omega}|\nabla c|^4dxdt \leq C_T.
\end{equation}
From the energy estimate in Lemma \ref{a-priori} we have
$\int_0^T\int_{\Omega}|\nabla \sqrt[4]{c}|^4 dxdt \leq C_T$.
Also since $\|c(t)\|_{L^\infty(\Omega)}$ is bounded, thanks to Lemma \ref{L1Linf}, the desired inequality \eqref{crucial} follows immediately. Applying Gronwall's inequality to \eqref{n_Lp} and taking into account\eqref{crucial}, we obtain that
	\begin{equation}\label{73}
	\|n\|_{L^\infty(0,T;L^p(\Omega))} \leq C_{p,T}
	\end{equation}
	for all $1\leq p < +\infty$. 

\medskip
The next step is to find a uniform bound for $\|A^\alpha u(t)\|_{L^2_\sigma(\Omega)}$, where $\frac{d}{4}<\alpha<1$. This can be done similarly as in \cite[Eq.~(4.19), pages 339-340]{Win12}. In \cite{Win12} it was shown that $\|A^\alpha u(t)\|_{L^2_\sigma(\Omega)}$ can be bounded if $\|n\|_{L^2(\Omega)}$ and $\|\nabla \varphi\|_{L^\infty(\Omega)}$ are bounded. Comparing to \cite{Win12}, we only have assumed that $\nabla\varphi\in L^\rho(\Omega)$ for $\rho > 6$. However, we can apply his calculations for $\tilde n:= n|\nabla\varphi|$, which is uniformly bounded in $L^2(\Omega)$ thanks to Young's inequality and \eqref{73}, and $\nabla\varphi/|\nabla\varphi|\in L^\infty(\Omega)$. Using $\alpha>\frac d4$ yields that $u$ is uniformly bounded thanks to Sobolev embeddings.

\medskip
Now we can proceed as in the proof of Lemmas 4.2 -- 4.4 of \cite{Bra17} to obtain that 
	\[c\in L^\infty((0,T);W^{1,10}(\Omega))\cap L^{10}((0,T);W^{2,10}(\Omega))\]
	and
	\[n\in L^\infty(0,T;W^{1,8}(\Omega)).\]
	These estimates are enough to see that the solution does not blow up and therefore $T_{\mathrm{max}}=\infty$.
\end{proof}
\subsection{In three dimensions}
In this section, we will again denote by $(n^{\varepsilon,m}, c^{\varepsilon,m}, u^{\varepsilon,m})$ the global classical solution to \eqref{Smod} for each $\varepsilon>0$ and $\mathbb N \ni m < \infty$. The main task is to study the limit $\varepsilon\to 0$ and $m\to \infty$. We first have the following uniform estimates.
\begin{lemma}\label{interpolation}
	We have
	\begin{equation}\label{n53}
		\{n^{\epsilon,m}\} \quad \text{ is bounded in } \quad L^{\frac 53}(Q_T),
	\end{equation}	
	and
	\begin{equation}\label{u103}
		\{u^{\epsilon,m}\} \quad \text{ is bounded in } \quad L^{\frac{10}{3}}(Q_T)
	\end{equation}
	uniformly in $\epsilon>0$ and $m>0$.
\end{lemma}
\begin{proof}
	From Lemma \ref{a-priori} we have
	\begin{equation*}
		\{\sqrt{n^{\epsilon,m}}\} \quad \text{ is bounded in } \quad L^\infty(0,T;L^2(\Omega))\cap L^2(0,T;H^1(\Omega)).
	\end{equation*}
	Since $d =  3$, $H^1(\Omega)\hookrightarrow L^6(\Omega)$ continuously. Moreover, an interpolation inequality gives
	\begin{equation*}
		L^\infty(0,T;L^2(\Omega))\cap L^2(0,T;L^6(\Omega)) \hookrightarrow L^{\frac{10}{3}}(Q_T)
	\end{equation*}
	continuously. Therefore $\{\sqrt{n^{\epsilon,m}}\}$ is bounded in $L^{\frac{10}{3}}(Q_T)$, which implies \eqref{n53}. The bound \eqref{u103} is proved similarly thanks to the fact that $\{u^{\epsilon,m}\}$ is bounded in $L^\infty(0,T;L^2(\Omega))\cap L^2(0,T;H^1(\Omega))$ which is followed from Lemma \ref{a-priori}.
\end{proof}
\begin{lemma}\label{limits}
	As $\epsilon\to 0$ and $m\to \infty$, up to a subsequence, we have the following convergences
	\begin{equation}\label{conv-n}
	n^{\epsilon,m} \rightarrow n \quad \text{ strongly in } L^{\frac 53-}(Q_T) \quad \text{ and weakly in } L^{\frac 54}(0,T;W^{1,\frac 54}(\Omega)),
	\end{equation}
	\begin{equation}\label{conv-c}
	c^{\epsilon, m} \rightarrow c \quad \text{ strongly in } L^{\infty-}(Q_T)  \quad \text{ and weakly in } L^4(0,T;W^{1,4}(\Omega)),
	\end{equation}
	\begin{equation}\label{conv-u}
	u^{\epsilon, m} \rightarrow u \quad \text{ strongly in } L^{\frac{10}{3}-}(Q_T) \quad \text{ and weakly in } L^2(0,T;H^1(\Omega)).
	\end{equation}
	Here we write $f^{\epsilon,m} \to f$ in $L^{p-}(Q_T)$ if $f^{\epsilon,m} \to f$ in $L^q(Q_T)$ for all $1\leq q < p$.
\end{lemma}
\begin{proof}We will prove the convergences \eqref{conv-n}, \eqref{conv-c} and \eqref{conv-u} separately. 

\medskip
	\noindent{\it Convergence of $n^{\epsilon,m}$.}
	From Lemma \ref{a-priori}, we have
	\begin{equation}\label{est_n}
	\|n^{\epsilon,m}\|_{L^{\frac 53}(Q_T)} + \int_0^T\int_{\Omega}\frac{|\nabla \sqrt{n^{\epsilon,m}}|^2}{n^{\epsilon,m}}dxdt \leq C_T.
	\end{equation}
	By H\"older's inequality we can estimate
	\begin{align}\label{grad_n}
	\int_0^T\int_{\Omega}|\na n^{\epsilon,m}|^{\frac 54}dxdt &= \int_0^T\int_{\Omega}\left(\frac{|\na n^{\epsilon,m}|^2}{n^{\epsilon,m}}\right)^{\frac 58}(n^{\epsilon,m})^{\frac 58}dxdt\\
	&\leq \left(\int_0^T\int_{\Omega}\frac{|\na n^{\epsilon,m}|^2}{n^{\epsilon,m}}dxdt\right)^{\frac 58}\left(\int_0^T\int_{\Omega}(n^{\epsilon,m})^{\frac 53}dxdt\right)^{\frac 38} \leq C_T.\nonumber
	\end{align}
	By testing the equation of $n^{\epsilon,m}$ with a smooth test function $\psi\in C^\infty_0(\Omega\times[0,T))$ we have
	\begin{equation}\label{weak-n}
	\begin{aligned}
	&-\int_0^T\int_{\Omega}n^{\epsilon,m} \psi_t dxdt - \int_{\Omega}n^{\epsilon,m}_0\psi(\cdot,0)dx\\
	&= -\int_0^T\int_{\Omega}\na n^{\epsilon,m}\na \psi dxdt - \int_0^T\int_{\Omega}n^{\epsilon,m}\na c^{\epsilon,m} \na\psi dxdt\\
	&+ \int_0^T\int_{\Omega}n^{\epsilon,m}u^{\epsilon,m}\cdot\na\psi dxdt+ \epsilon\int_0^T\int_{\Omega}(n^{\epsilon,m} - (n^{\epsilon,m})^3)\psi dxdt.
	\end{aligned}
	\end{equation}
	From \eqref{grad_n} we can estimate
	\begin{equation*}
	\left|\int_0^T\int_{\Omega}\nabla n^{\epsilon,m}\na \psi dxdt\right| \leq C_T\|\psi\|_{L^{5}(Q_T)}.
	\end{equation*}
	Lemma \ref{a-priori} gives
	\begin{equation*}
	\begin{aligned}
	&\left|\int_0^T\int_{\Omega}n^{\epsilon,m}\na c^{\epsilon,m} \na\psi dxdt\right|\\
	&\leq \int_0^T\int_{\Omega}|\na c^{\epsilon,m}|\sqrt{n^{\epsilon,m}} \sqrt{n^{\epsilon,m}} |\na\psi|dxdt\\
	&\leq \left(\int_0^T\int_{\Omega}|\na c^{\epsilon,m}|^2n^{\epsilon,m} dxdt\right)^{\frac 12}\left(\int_0^T\int_{\Omega}|n^{\epsilon,m}|^{2}dxdt\right)^{\frac 12}\|\na\psi\|_{L^\infty(Q_T)}\\
	&\leq C_T\|\psi\|_{L^\infty(0,T;W^{2,4}(\Omega))},
	\end{aligned}
	\end{equation*}
	thanks to the estimates in three dimensions $\|\nabla \psi\|_{L^\infty(\Omega)}\leq C\|\nabla \psi\|_{W^{1,4}(\Omega)} \leq C\|\psi\|_{W^{2,4}(\Omega)}$. Using the same idea we estimate
	\begin{align*}
	\left|\int_0^T\int_{\Omega}n^{\epsilon,m}u^{\epsilon,m}\cdot \na\psi dxdt \right| &\leq \|\psi\|_{L^\infty(Q_T)}\int_0^T\int_{\Omega}|n^{\epsilon,m}||u^{\epsilon,m}|dxdt\\
	&\leq C\|\psi\|_{L^\infty(0,T;W^{2,4}(\Omega))}\|u^{\epsilon,m}\|_{L^{\frac{10}{3}}(Q_T)}\|n^{\epsilon,m}\|_{L^{\frac{10}{7}}(Q_T)}\\
	&\leq C\|\psi\|_{L^\infty(0,T;W^{2,4}(\Omega))},
	\end{align*}
	thanks to \eqref{est_n} and the fact that $\frac{10}{7} < \frac{5}{3}$. From \eqref{nl} it follows that
	\begin{equation*}
	\left|\epsilon\int_0^T\int_{\Omega}(n^{\epsilon,m} - (n^{\epsilon,m})^3) \psi dxdt\right| \leq C_T\|\psi\|_{L^\infty(Q_T)}.
	\end{equation*}
	Combining these estimates we obtain
	\begin{equation}\label{est_nt}
	\{\partial_tn^{\epsilon,m}\} \quad \text{ is bounded } \quad \text{ in } \quad L^1(0,T;(W^{2,4}(\Omega))^*).
	\end{equation}
	
	From \eqref{est_n}, \eqref{grad_n} and \eqref{est_nt} it follows from Aubin-Lions lemma that 
	\begin{equation*}
	n^{\epsilon,m} \to n \quad \text{ in } \quad L^{\frac 54}(Q_T)
	\end{equation*}
	as $\epsilon\to 0$ and $m\to\infty$ (up to a subsequence). Moreover, since $\{n^{\epsilon,m}\}$ is bounded in $L^{\frac 53}(Q_T)$, thanks to Lemma \ref{interpolation}, we have in fact
	\begin{equation*}
	n^{\epsilon,m} \to n \quad \text{ in } \quad L^{\frac 53-}(Q_T).
	\end{equation*}
	This convergence and \eqref{grad_n} give the convergence for $n$ as in \eqref{conv-n}.
	
	\medskip
	\noindent{\it Convergence of $c^{\epsilon,m}$.}
	From $\|c^{\epsilon,m}\|_{L^\infty(Q_T)} \leq C$ and $\int_0^T\int_{\Omega}\frac{|\na c^{\epsilon,m}|^4}{(c^{\epsilon,m})^3}dxdt \leq C_T$ we have
	\begin{equation}\label{grad-c}
	\int_0^T\int_{\Omega}|\na c^{\epsilon,m}|^4dxdt \leq C\int_0^T\int_{\Omega}\frac{|\na c^{\epsilon,m}|^4}{(c^{\epsilon,m})^3}dxdt \leq C_T.
	\end{equation}
	By testing the second equation in \eqref{Smod} with a smooth test function $\psi \in C^\infty_0(\Omega\times[0,T))$ we have
	\begin{align}\label{weak-c}
	&-\int_0^T\int_{\Omega}c^{\epsilon,m} \psi_t dxdt - \int_{\Omega}c^{\epsilon,m}_0\psi(\cdot,0)dx\nonumber\\
	&= -\int_0^T\int_{\Omega}\na c^{\epsilon,m} \na \psi dxdt + \int_0^T\int_{\Gamma}g(x)(\gamma - c^{\epsilon,m})\psi \dH dt\nonumber\\
	&+\int_0^T\int_{\Omega}c^{\epsilon,m}u^{\epsilon,m}\na\psi dxdt -\int_0^T\int_{\Omega}n^{\epsilon,m}c^{\epsilon,m}\psi dxdt.
	\end{align}
	We have the following estimates
	\begin{equation*}
	\left|\int_0^T\int_{\Omega}\na c^{\epsilon,m}\na \psi dxdt\right| \leq \|\na c^{\epsilon,m}\|_{L^4(Q_T)}\|\na\psi\|_{L^{\frac{4}{3}}(Q_T)} \leq C\|\psi\|_{L^{\frac 43}(0,T;W^{1,\frac 43}(\Omega))},
	\end{equation*}
	\begin{align*}
	\left|\int_0^T\int_{\Gamma}\kappa(x)(\gamma - c^{\epsilon,m})\psi \dH dt\right| &\leq C(1+\|c^{\epsilon,m}\|_{L^\infty((0,T)\times \Gamma)})\int_0^T\int_{\Gamma}|\psi|\dH dt\\
	&\leq C_T\|\psi\|_{L^1(0,T;W^{1,1}(\Omega))},
	\end{align*}
	\begin{align*}
	\left|\int_0^T\int_{\Omega}c^{\epsilon,m}u^{\epsilon,m}\na\psi dxdt\right| &\leq \|c^{\epsilon,m}\|_{L^\infty(Q_T)}\|u^{\epsilon,m}\|_{L^2(Q_T)}\|\na\psi\|_{L^2(Q_T)}\\
	&\leq C_T\|\psi\|_{L^2(0,T;H^1(\Omega))}.
	\end{align*}
	and
	\begin{align*}
	\left|\int_0^T\int_{\Omega}n^{\epsilon,m}c^{\epsilon,m}\psi dxdt\right| &\leq \|c^{\epsilon,m}\|_{L^\infty(Q_T)}\|n^{\epsilon,m}\|_{L^{\frac 53}(Q_T)}\|\psi\|_{L^{\frac 52}(Q_T)}\\
	&\leq C\|\psi\|_{L^{\frac 52}(0,T;L^{\frac 52}(\Omega))}.
	\end{align*}
	Therefore
	\begin{equation*}
	\left|\int_0^T\int_{\Omega}\partial_t c^{\epsilon,m}\psi dxdt\right| \leq C_T\|\psi\|_{L^{\frac 52}(0,T;W^{1,\frac 52}(\Omega))}.
	\end{equation*}
	thus
	\begin{equation*}
	\{\partial_t c^{\epsilon,m} \} \quad \text{ is bounded in } \quad L^{\frac 53}(0,T;(W^{1,\frac 52}(\Omega))^*).
	\end{equation*}
	Combining this with \eqref{grad-c} and the uniform bound of $c^{\epsilon,m}$, it follows from the Aubin-Lions lemma that
	\begin{equation*}
	c^{\epsilon,m} \to c \quad \text{ in } \quad L^4(Q_T)
	\end{equation*}
	as $\epsilon\to 0$ and $m\rightarrow \infty$, and consequently \eqref{conv-c} thanks to the boundedness of $c^{\epsilon,m}$ in $L^\infty(Q_T)$.
	
	\medskip
	\noindent{\it Convergence of $u^{\epsilon,m}$.} Testing the equation of $u^{\epsilon,m}$ in \eqref{Smod} with $\psi \in C^\infty_0(\Omega\times [0,T))^3$ we get
	\begin{align}\label{weak-u}
	&-\int_0^T\int_{\Omega}u^{\epsilon,m}\cdot \psi_t dxdt - \int_{\Omega}u^{\epsilon,m}_0\cdot\psi(\cdot,0)dx\nonumber\\
	&= \int_0^T\int_{\Omega}\na u^{\epsilon,m} \cdot \na\psi dxdt - \int_0^T\int_{\Omega}\mathcal P^mB(u^{\epsilon,m},u^{\epsilon,m})\psi dxdt- \int_0^T\int_{\Omega}\mathcal{P}^m(n^{\epsilon,m}\na\varphi)\psi dxdt. 
	\end{align}
	We estimate the terms on the right hand side as following
	\begin{align*}
	\left|\int_0^T\int_{\Omega}\na u^{\epsilon,m} \cdot \na\psi dxdt \right| \leq \|\na u^{\epsilon,m}\|_{L^2(Q_T)}\|\na\psi\|_{L^2(Q_T)} \leq C_T\|\psi\|_{L^2(0,T;H^1(\Omega))},
	\end{align*}
	\begin{align*}
	\left|\int_0^T\int_{\Omega}\mathcal P^mB(u^{\epsilon,m},u^{\epsilon,m})\psi dxdt\right| \leq \|u^{\epsilon,m}\|_{L^{\frac{10}{3}}(Q_T)}^2\|\na \psi\|_{L^{\frac 52}(Q_T)} \leq C_T\|\psi\|_{L^{\frac 52}(0,T;W^{1,\frac 52}(\Omega))},
	\end{align*}
	and
	\begin{align*}
	\left|\int_0^T\int_{\Omega}\mathcal{P}^m(n^{\epsilon,m}\na\varphi)\psi dxdt\right| \leq \|\varphi\|_{W^{1,\rho}(\Omega)}T\|n^{\epsilon,m}\|_{L^{\frac 53}(Q_T)}\|\psi\|_{L^{\frac{5\rho}{2\rho - 5}}(Q_T)} \leq C_T\|\psi\|_{L^5(Q_T)},
	\end{align*}
	where we used $\frac{5\rho}{2\rho - 5} < 5$ at the end since $\rho > 6$. Therefore
	\begin{equation*}
	\{\partial_t u^{\epsilon,m} \} \quad \text{ is bounded in } \quad L^{\frac 54}(0,T;(W^{1,5}(\Omega))^*).
	\end{equation*}
	Now the Aubin-Lions gives us the strong convergence
	\begin{equation*}
	u^{\epsilon,m} \to u \quad \text{ in } \quad L^2(Q_T)
	\end{equation*}
	as $\epsilon\to 0$ and $m\to \infty$ (up to a subsequence). Finally, \eqref{conv-u} follows from the fact that $\{u^{\epsilon,m}\}$ is bounded in $L^{\frac{10}{3}}(Q_T)$ from Lemma \ref{interpolation}.
\end{proof}

We need one more preparation which was proved in \cite[Lemma 4.11]{Bra17}.
\begin{lemma}\cite[Lemma 4.11]{Bra17}\label{conv_logistic}
	The sequence $\{\epsilon(n^{\epsilon,m} - (n^{\epsilon,m})^3)\}$ is weakly precompact in $L^1(Q_T)$.
\end{lemma}
We are now ready to prove Theorem \ref{thm:main3D}.
\begin{proof}[Proof of Theorem \ref{thm:main3D}]
	It is sufficient to show that the limits function $(n,c,u)$ obtained in Lemma \ref{limits} is a weak solution in the sense of Definition \ref{weak_sol}. In order to do that, we need to take care of the limits $\epsilon\to 0$ and $m\to \infty$ in \eqref{weak-n}, \eqref{weak-c} and \eqref{weak-u}.
	
	\medskip
	For the first term on the right hand side of \eqref{weak-n}, we write $\na n^{\epsilon,m} = 2\na\sqrt{n^{\epsilon,m}}\sqrt{n^{\epsilon,m}}$ and use $\na\sqrt{n^{\epsilon,m}} \rightharpoonup \na\sqrt{n}$ in $L^2(Q_T)$ and $\sqrt{n^{\epsilon,m}}\to \sqrt{n}$ in $L^2(Q_T)$ we get that $\na n^{\epsilon,m} \rightharpoonup \na n$ weakly in $L^2(Q_T)$, hence
	\begin{equation*}
	\int_0^T\int_{\Omega}\na n^{\epsilon,m}\cdot \na \psi dxdt \xrightarrow{\epsilon\to 0, \; m\to\infty} \int_0^T\int_{\Omega}\na n\cdot \na \psi dxdt.
	\end{equation*}
	From $\na c^{\epsilon,m} \rightharpoonup \na c$ weakly in $L^4(Q_T)$ and $n^{\epsilon,m} \to n$ strongly in $L^{\frac 43}(Q_T)$ it follows
	\begin{equation*}
	\int_0^T\int_{\Omega}n^{\epsilon,m}\na c^{\epsilon,m}\cdot \na\psi dxdt \xrightarrow{\epsilon\to 0, \; m\to\infty} \int_0^T\int_{\Omega}n\na c\cdot \na\psi dxdt.
	\end{equation*}
	From \eqref{conv-n} and \eqref{conv-u} we have $n^{\epsilon,m} \to n$ strongly in $L^{\frac{20}{13}}(Q_T)$ and $u^{\epsilon,m}\to u$ strongly in $L^{\frac{20}{7}}(Q_T)$, and thus $n^{\epsilon,m}u^{\epsilon,m} \to nu$ strongly in $L^1(Q_T)$ and consequently
	\begin{equation*}
	\int_0^T\int_{\Omega}n^{\epsilon,m}u^{\epsilon,m}\cdot\na\psi dxdt \xrightarrow{\epsilon\to 0, \; m\to\infty} \int_0^T\int_{\Omega}nu\cdot\na\psi dxdt.
	\end{equation*}
	The convergence of the last term
	\begin{equation*}
	\int_0^T\int_{\Omega}\epsilon(n^{\epsilon,m} - (n^{\epsilon,m})^3)\psi dxdt \xrightarrow{\epsilon\to 0,\; m\to\infty} 0
	\end{equation*}
	follows from Lemma \ref{conv_logistic}.
	
	\medskip
	Similarly, the convergence of all terms in \eqref{weak-c} holds thanks to \eqref{conv-n}-\eqref{conv-c}-\eqref{conv-u} and the interpolation inequality
	\begin{equation*}
	\int_0^T\int_{\Gamma}|c^{\epsilon,m} - c|^2\dH dt \leq C\|c^{\epsilon,m} - c\|_{L^2(0,T;H^1(\Omega))}\|c^{\epsilon,m} - c\|_{L^2(0,T;L^2(\Omega)}.
	\end{equation*}
	
	\medskip
	All the terms in \eqref{weak-u} can be treated similarly thanks to \eqref{conv-n}-\eqref{conv-u}.
	
	The uniform bound of the energy will be proved in Proposition \ref{bound_energy}.
\end{proof}

	\begin{proposition}[Uniform-bound of the energy]\label{bound_energy}
		We have the following bound of the energy
		\begin{equation*}
		\sup_{t\in [0,\infty)}\left(\int_{\Omega}n(t)\log n(t)dx + \| \na\sqrt c(t)\|_{L^2(\Omega)}^2 + \|u(t)\|_{L^2(\Omega)}^2\right) \leq C.
		\end{equation*}
	\end{proposition}
	\begin{proof}
		From \eqref{k1}, we deduce by choosing $\delta$ small enough that
		\begin{equation}\label{k2}
			\mathcal F(n,c,u)(t) + C\int_s^t\mathcal E(n,c,u)(r)dr \leq \mathcal F(n,c,u)(s) + C(t-s) + C\int_s^t\int_{\Omega}|\na\sqrt c|^2dxdr
		\end{equation}
		where
		\begin{equation}\label{def_E}
			\mathcal E(n,c,u) = \int_{\Omega}\left(|\na\sqrt{n}|^2 + c|\na^2\log c|^2  + \frac{|\na c|^4}{c^3} + |\na \sqrt{c}|^2n + \|u\|_{H^1(\Omega)}^2 \right)dx.
		\end{equation}
		Looking at \eqref{k2}, it becomes clear that that last term on the right-hand side is troublesome as it prevents to obtain uniform bound in time of the energy while applying Gronwall's lemma. To overcome this difficulty, we introduce an additional energy, namely $\int_{\Omega}s^\infty(c|\gamma)dx$, where we recall that $s^\infty$ is defined in \eqref{entropy}. Remark that here we consider an extension of the surface function $\gamma: \Gamma \to \mathbb R$ into the (with a slight abuse of notation) $\gamma:\overline{\Omega}\to \infty$, with $\gamma \in H^1(\overline{\Omega})\cap L^\infty(\overline{\Omega})$. Moreover, thanks to \eqref{cond_data}, $\gamma(x) \geq \underline{\gamma} >0$ for all $x\in\overline{\Omega}$. We now show that 
		\begin{align}\label{k3}
		\frac d{dt}\int_{\Omega}s^\infty(c|\gamma)dx+4\int_{\Omega}\left|\nabla\sqrt{c}\right|^2dx+\int_{\Gamma}\kappa s^\infty(\gamma|c) d\mathcal H^{d-1}_x\leq C.
		\end{align}
		Indeed, by direct computations, we have
		\begin{align*}
		\frac d{dt}\int_{\Omega}s^\infty(c|\gamma)dx+4\int_{\Omega}|\na\sqrt{c}|^2dx
		&=
		\int_{\Omega}\partial_tc\log\frac{c}\gamma dx	+\int_{\Omega}\nabla c\cdot \nabla\log cdx
		\\&= \int_{\Gamma}\kappa (\gamma- c) \log\frac{c}{\gamma}d\mathcal H^{d-1}_x-\int_{\Omega}cn \log\frac{c}{\gamma}dx.
		\end{align*}
		We now observe the following two  identities
		\begin{align*}
		(\gamma- c) \log\frac{c}{\gamma}&=
		-c\log\frac{c}{\gamma}+c-\gamma
		+\gamma\log\frac{c}{\gamma}+\gamma-c	\\&=
		-s^\infty(c|\gamma)- s^\infty(\gamma|c),
		\end{align*}
		and
		\begin{align*}
		-cn \log\frac{c}{\gamma}
		&=-n\left(c\log\frac{c}{\gamma}-c+\gamma\right)+n(\gamma-c)
		=n(\gamma-c)-ns^\infty(c|\gamma).
		\end{align*}
		Combing these calculations and using $s^\infty(x|y) \geq 0$ lead to
		\begin{align*}
		\frac d{dt}\int_{\Omega}s^\infty(c|\gamma)dx+4\int_{\Omega}|\na\sqrt c|^2 dx
		+\int_{\Gamma}\kappa s^\infty(\gamma|c) d\mathcal H^{d-1}_x\\
		\leq \int_{\Omega}n(\gamma-c)dx \leq \|\gamma\|_{L^\infty(\Omega)}\|n\|_{L^1(\Omega)} \leq C,
		\end{align*}
		due to the non-negativity of $n$ and $c$, and $\|n\|_{L^1(\Omega)}\leq C$. 

		\medskip		
		By multiplying \eqref{k3} by a large constant $L>0$, integrating the resultant on $(s,t)$, and adding the obtained inequality to \eqref{k2}, we get
		\begin{equation}\label{k4}
			\mathfrak X(n,c,u)(t) + C\int_s^t \mathfrak Z(n,c,u)(r)dr \leq \mathfrak X(n,c,u)(s) + C(t-s)
		\end{equation}
		where
		\begin{align}\label{X}
			\mathfrak X(n,c,u)(t) = \mathcal F(n,u,c)(t) + L\int_{\Omega}s^{\infty}(c|\gamma)(t)dx
			\end{align}
			and
		\begin{align*}
			\mathfrak Z(n,c,u)(r) = \mathcal E(n,c,u)(r) + \int_{\Omega}|\na \sqrt c(r)|^2dx + \int_{\Gamma}\kappa s^\infty(\gamma|c)\dH.
		\end{align*}
		We will now prove for some constants $\lambda>0$ and $C>0$ that
		\begin{equation}\label{eede}
			\mathfrak Z(n,c,u) \geq \lambda \mathfrak X(n,c,u) - C.
		\end{equation}
		From \eqref{def_F} and \eqref{def_E}, to obtain \eqref{eede}, it remains to show that
		\begin{equation}\label{k5}
			\int_{\Omega}|\na\sqrt n|^2dx \geq \lambda_1\int_{\Omega}n\log n dx - C
		\end{equation}
		and 
		\begin{equation}\label{k6}
			\int_{\Omega}|\na\sqrt{c}|^2dx \geq \lambda_2\int_{\Omega}s^{\infty}(c|\gamma) - C.
		\end{equation}
		By the Logarithmic-Sobolev inequality we have, where $\overline{n} = \frac{1}{|\Omega|}\int_{\Omega}n(x)dx$,
		\begin{align*}
			\int_{\Omega}|\na \sqrt n|^2dx &= \frac 14\int_{\Omega}\frac{|\na\sqrt n|^2}{n}dx \geq \frac 14 C_{LSI}\int_{\Omega}n\log\frac{n}{\overline{n}}dx\\
			&= \frac 14C_{LSI}\int_{\Omega}n\log n dx - \frac 14 C_{LSI}\log \overline{n}\|n\|_{L^1(\Omega)},
		\end{align*}
		hence \eqref{k5}, thanks to $\|n\|_{L^1(\Omega)} \leq C$. Similarly
		\begin{align*}
			\int_{\Omega}|\na \sqrt c|^2dx &\geq \frac 14C_{LSI}\int_{\Omega}c\log c\,dx - \frac 14C_{LSI}\log \overline{c}\|c\|_{L^1(\Omega)}\\
			&= \frac 14C_{LSI}\int_{\Omega}s^\infty(c|\gamma)dx + \frac 14 C_{LSI}\int_{\Omega}(c\log \gamma + c - \gamma)dx - \frac 14C_{LSI}\frac{1}{|\Omega|}\|c\|_{L^1(\Omega)}^2
		\end{align*}
		hence \eqref{k6}, due to $\|c\|_{L^\infty(\Omega)} \leq C$, $\|\gamma\|_{L^\infty(\Omega)} \leq C$ and $\gamma(x) \geq \underline{\gamma} > 0$. We have proved \eqref{k5} and \eqref{k6}, and consequently \eqref{eede}. Using \eqref{eede} in \eqref{k4}, it follows that
		\begin{equation*}
		\mathfrak{X}(n,c,u)(t) + C\int_s^t \mathfrak{X}(n,c,u)(r)dr \leq \mathfrak{X}(n,c,u)(s) + C(t-s).
		\end{equation*}
		Thus for some suitable constant $K$ and $\Theta(r) = \mathfrak{X}(n,c,u)(r) - K$,
		\begin{equation*}
		\Theta(t) + C\int_s^t\Theta(r)dr 
		\leq \Theta(s).
		\end{equation*}
		Defining
		\begin{equation*}
		\Xi(s) = \int_s^t \Theta(r)dr,
		\end{equation*}
		we have
		\begin{equation*}
		\Xi'(s) = -\Theta(s) \leq \Theta(t) - C\Xi(s)
		\end{equation*}
		and consequently
		\begin{equation*}
		\left(e^{Cs}\Xi(s)\right)' + e^{Cs}\Theta(t)\leq 0.
		\end{equation*}
		Integrating this from $0$ to $t$, noting that $\Xi(t) = 0$, gives
		\begin{equation*}
		-\Xi(0) + \Theta(t)\frac{e^{Ct}-1}{C}\leq 0.
		\end{equation*}
		From this we have
		\begin{equation*}
		\Theta(t)\frac{e^{Ct}-1}{C} \leq \Xi(0) = \int_0^t\Theta(r)dr \leq \frac{\Theta(0) -\Theta(t)}{C}
		\end{equation*}
		and therefore
		\begin{equation*}
		\Theta(t) \leq e^{-Ct}\Theta(0).
		\end{equation*}
		Replacing $\Theta(t) = \mathfrak{X}(n,c,u)(t)-K$ we finally obtain
		\begin{equation*}
		\mathfrak{X}(n,c,u)(t) \leq K + e^{-Ct}(\mathfrak X(n_0,c_0,u_0) - K) \leq C
		\end{equation*}
		for all $t>0$, which proves our claim thanks to \eqref{X} and \eqref{def_F}.
	\end{proof}

	\medskip
	\noindent{\bf Acknowledgement.} The authors gratefully acknowledge the support of the Hausdorff Research Institute for Mathematics (Bonn) through the Junior Trimester Kinetic Theory. This work is partially supported by NAWI Graz,  the International Research Training Group IGDK 1754 and by the Austrian Science Fund (FWF) project F 65.

\end{document}